\documentclass[10.5pt,a4paper]{article}
\usepackage{graphicx,latexsym,euscript,makeidx,color,bm}
\usepackage{amsmath,amsfonts,amssymb,amsthm,thmtools,mathrsfs,enumerate}
\usepackage[colorlinks,linkcolor=blue,anchorcolor=green,citecolor=red]{hyperref}
\usepackage{graphicx,tikz,exscale,pagecolor}  

\usepackage{booktabs,subfigure,xcolor}

\usepackage{geometry}
   \def\cB{{\cal B}}

\def\dbE{\mathbb{E}}   \def\cE{{\cal E}}  
\def\dbF{\mathbb{F}}   \def\cF{{\cal F}}  
     
   \def\cH{{\cal H}}

   \def\cM{{\cal M}}

\def\dbP{\mathbb{P}}   \def\cP{{\cal P}}  
     
\def\dbR{\mathbb{R}}     
     
 \def\sT{\mathscr{T}}

\def\ss{\smallskip}      \def\lt{\left}       \def\hb{\hbox}
\def\ms{\medskip}        \def\rt{\right}      \def\ae{\hbox{\rm a.e.}}
\def\bs{\bigskip}        \def\lan{\langle}    \def\as{\hbox{\rm a.s.}}
\def\ds{\displaystyle}   \def\ran{\rangle}    
\def\ts{\textstyle}         
\def\no{\noindent}          
\def\ns{\noalign{\ss}}     \def\esssup{\mathop{\rm esssup}}
     
\def\rf{\eqref}            
         \def\hp{\hphantom}
\def\deq{\triangleq}     \def\({\Big (}       \def\nn{\nonumber}
\def\les{\leqslant}      \def\){\Big )}       
\def\ges{\geqslant}      \def\[{\Big[}        \def\cl{\overline}
\def\ti{\tilde}          \def\]{\Big]}        
\def\wt{\widetilde}      \def\q{\quad}        
         \def\qq{\qquad}      \def\1n{\negthinspace}
\def\cd{\cdot}           \def\2n{\1n\1n}      
\def\cds{\cdots}         \def\3n{\1n\2n}

\def\a{\alpha}        \def\G{\Gamma}   \def\g{\gamma}   \def\Om{\Omega}   \def\om{\omega}
\def\b{\beta}         \def\D{\Delta}           \def\p{\phi}
\def\z{\zeta}           \def\th{\theta}    \def\si{\sigma}
\def\e{\varepsilon}     \def\l{\lambda}        
    \def\t{\tau}     \def\f{\varphi}  \def\i{\infty}   

\def\ba{\begin{array}}                \def\ea{\end{array}}
\def\bel{\begin{equation}\label}      \def\ee{\end{equation}}

\newtheoremstyle{indented}{}{}{\it}{\parindent}{\bfseries}{.}{.5em}{}
\theoremstyle{indented}

\newtheorem{theorem}{Theorem}[section]
\newtheorem{definition}[theorem]{Definition}
\newtheorem{proposition}[theorem]{Proposition}

\newtheorem{lemma}[theorem]{Lemma}
\newtheorem{remark}[theorem]{Remark}
\newtheorem{example}[theorem]{Example}
\newtheorem{taggedassumptionx}{}
\newenvironment{taggedassumption}[1]
 {\renewcommand\thetaggedassumptionx{#1}\taggedassumptionx}
 {\endtaggedassumptionx}

\makeatletter
   
   \@addtoreset{equation}{section}
\makeatother
  \allowdisplaybreaks[4]

\begin{document}

\title{\bf Recursive Utility Processes, Dynamic Risk Measures \\ and Quadratic Backward\\ Stochastic Volterra Integral Equations}

\author{
Hanxiao Wang\thanks{School of Mathematical Sciences, Fudan University, Shanghai 200433, China
                    ({\tt hxwang14@fudan.edu.cn}). This author is supported in part by the
                    China Scholarship Council, while visiting University of Central Florida.} \and
Jingrui Sun\thanks{Corresponding author. Department of Mathematics, Southern University of Science and Technology, Shenzhen, Guangdong 518055, China
                   ({\tt sunjr@sustech.edu.cn} ).} \and
Jiongmin Yong\thanks{Department of Mathematics, University of Central Florida, Orlando, FL 32816, USA
                     ({\tt jiongmin.yong@ucf.edu}). This author is supported in part by NSF Grant DMS-1812921.}
}

\maketitle

\no\bf Abstract. \rm For an $\cF_T$-measurable payoff of a European type contingent claim, the recursive utility process/dynamic risk measure can be described by the adapted solution to a backward stochastic differential equation (BSDE). However, for an $\cF_T$-measurable stochastic process (called a position process, not necessarily $\dbF$-adapted), mimicking BSDE's approach will lead to a time-inconsistent recursive utility/dynamic risk measure. It is found that a more proper approach is to use the adapted solution to a backward stochastic Volterra integral equation (BSVIE). The corresponding notions are called equilibrium recursive utility and equilibrium dynamic risk measure, respectively. Motivated by this, the current paper is concerned with BSVIEs whose generators are allowed to have quadratic growth (in $Z(t,s)$). The existence and uniqueness for both the so-called adapted solutions and adapted M-solutions are established. A comparison theorem for adapted solutions to the so-called Type-I BSVIEs is established as well. As consequences of these results, some general continuous-time equilibrium dynamic risk measures and equilibrium recursive utility processes are constructed.

\ms

\no\bf Key words. \rm
Backward stochastic Volterra integral equation, quadratic generator, comparison theorem, equilibrium dynamic risk measure, equilibrium recursive utility process,
time-consistency.

\ms

\no\bf AMS subject classifications. \rm 60H10, 60H20, 91B30, 91B70, 91G80.

\section{Introduction}

Let $(\Om,\cF,\dbP)$ be a complete probability space on which a one-dimensional standard Brownian motion $W=\{W(t);0\les t<\i\}$ is defined, with $\dbF=\{\cF_t\}_{t\ges0}$ being the natural filtration of $W$ augmented by all the $\dbP$-null sets in $\cF$. Let $\xi$ be a (random) payoff at some future time $T$ of a certain European type contingent claim, and $c(\cd)$ be a consumption rate. Following \cite{Duffie-Epstein 1992}, we let $Y(\cd)$ solve the following equation:
\bel{BSDE1}Y(t)=\dbE_t\[\xi+\int_t^T\(f(c(s),Y(s))+A(Y(s))Z(s)^2\)ds\],\q t\in[0,T],\ee
hereafter, $\dbE_t[\,\cd\,]=\dbE[\,\cd\,|\,\cF_t]$ is the conditional expectation operator, and $f:\dbR\times\dbR\to\dbR$ is a given map, called the {\it aggregator},
$$Z(t)^2={d\over dt}\lan Y\ran(t),$$
with $t\mapsto\lan Y\ran(t)$ being the quadratic variation process of $Y(\cd)$, and $A(Y(t))$ is called the {\it variance multiplier}. Such defined $Y(\cd)$ is called a {\it recursive utility process} (which has also been called {\it stochastic differential utility process}) of the payoff $\xi$ and the consumption rate $c(\cd)$.  The main feature of such a process $Y(\cd)$ is that the current value $Y(t)$ depends on the future values $Y(s)$, $t<s\les T$ of the process. This notion was firstly introduced by Duffie and Epstein \cite{Duffie-Epstein 1992} in 1992.
It is easy to see that $(Y(\cd),Z(\cd))$ solves \rf{BSDE1} if and only if it is an adapted solution to the following backward stochastic differential equation (BSDE, for short):
\bel{BSDE2}Y(t)=\xi+\int_t^Tg(s,Y(s),Z(s))ds-\int_t^TZ(s)dW(s),\q t\in[0,T],\ee
with
\bel{g=z2}g(s,y,z)=f(c(s),y)+A(y)z^2.\ee
Thanks to the discovery of the relation between \rf{BSDE1} and \rf{BSDE2}, recursive utility process was later extended to the adapted solution of general BSDEs (see \cite{El Karoui-Peng-Quenez 1997,Lazrak-Quenez 2003,Lazrak 2004}).

\ms

Now, if instead of $\xi$, we have an $\cF_T$-measurable process $\psi(t)$, not necessarily $\dbF$-adapted, which is called a {\it position process} (see \cite{Riedel 2004} for a study of discrete-time cases). It could also be called an {\it anticipated wealth flow process}. For example, it could be an anticipated received dividend process of a stock (which depends on the uncertain performance of the company),  anticipated received mortgage payments (for a bank, say, with an uncertainty of default or prepayment), anticipated claim payments of an insurance policy, the random maintenance costs of an owned facility, etc. The feature of such kind of process is that at time $t$, the actually anticipated value of the process is not $\cF_t$-measurable. To ``calculate'' the recursive utility for such a process at the current time $t$, mimicking \rf{BSDE1}, we might formally solve the following BSDE:
%
%
%
\bel{BSDE4}Y(t;r)=\psi(t)+\int_r^Tg(s,Y(t;s),Z(t;s))ds-\int_r^TZ(t;s)dW(s),
\qq r\in[t,T],\ee
with the current time $t$ being a parameter. Intuitively, $Y(t;r)$ should represent the utility of the process $\psi(\cd)$ at a future time $r$, estimated/predicted at the current time $t$. Therefore, the utility at the current time $t$ should be given by $Y(t;t)$. However, by taking $r=t$ in the above, we obtain
\bel{BSDE5}Y(t;t)=\psi(t)+\int_t^Tg(s,Y(t;s),Z(t;s))ds-\int_t^TZ(t;s)dW(s),\q t\in[0,T],\ee
which is not an equation for the process $t\mapsto Y(t;t)$ since $Y(t;s)$ appears on the right-hand side of the above. A careful observation shows that $Y(t;r)$ obtained through \rf{BSDE4} has some time-inconsistent nature, by which we mean the following: If everything is ideal, the value $Y(t;r)$, which is supposed to be the utility of the process $\psi(\cd)$ at a future time $r$ estimated/predicted at the current time $t$ should be equal to $Y(r;r)$, the realistic utility at future time $r$. But this seems to have very little hope. In another word, $t\mapsto Y(t;t)$ determined by a family of BSDEs as above seems not to be a good description of the recursive utility process for the position process $\psi(\cd)$.

\ms

Suggested by \rf{BSDE4}--\rf{BSDE5}, we propose the following modified equation:
\bel{BSVIE1}Y(t)=\psi(t)+\int_t^Tg(s,Y(s),Z(t,s))ds-\int_t^TZ(t,s)dW(s),\qq t\in[0,T].\ee
Note that the above modification is simply to force $Y(t;s)=Y(s;s)$ in \rf{BSDE5}, then rename $Y(t;t)$ to be $Y(t)$. The advantage of such a modification is that as long as a solution $(Y(\cd),Z(\cd\,,\cd))$ of \rf{BSVIE1} exists, $Y(\cd)$ is time-consistent. Then, $Y(\cd)$ could serve as a good description of the recursive utility for the process $\psi(\cd)$ (by suitably selecting the aggregator $g(s,y,z)$). However, a couple of natural questions arise: (i) Is there any convincing mathematical justification for the model \rf{BSVIE1}, and (ii) By ``brutally'' forcing $Y(t;s)=Y(s;s)$, is the resulting equation \rf{BSVIE1}  well-posed? For question (i), we will sketch a convincing argument in the appendix at the end of the paper, justifying our modification. We will borrow some ideas from the study of time-inconsistent optimal control problems (\cite{Yong 2012}). For question (ii), it turns out that \rf{BSVIE1} is nothing but a so-called {\it backward stochastic Volterra integral equation} (BSVIE, for short), which has been studied since the early 2000 for various cases, and the current paper is actually a continuation of those investigations. With the well-posedness of \rf{BSVIE1} (see below for details), the map $t\mapsto Y(t)$ will be called an {\it equilibrium recursive utility process} of $\psi(\cd)$. Interestingly, our mathematical justification presented in the appendix will perfectly justify the word ``equilibrium''.

\ms

BSVIEs have been studied since 2002 (\cite{Lin 2002}). Let us now elaborate a little more on BSVIEs. Let
$$ g:[0,T]^2\times\dbR\times\dbR\times\dbR\times\Om \to\dbR,\qq\psi:[0,T]\times\Om \to\dbR$$
be two given random fields. We consider the following BSVIE:
\bel{bsvie-II} Y(t)=\psi(t) + \int_t^T g(t,s,Y(s),Z(t,s),Z(s,t))ds - \int_t^T Z(t,s)dW(s),\q t\in[0,T]. \ee
By an {\it adapted solution} to BSVIE \rf{bsvie-II}, we mean an $(\dbR\times\dbR)$-valued random field $(Y,Z)=\{(Y(t)$, $Z(t,s));0\les s,t\les T\}$ such that
\begin{enumerate}[(i)]
\item $Y(\cd)$ is $\dbF$-progressively measurable (not necessarily continuous),

\item for each fixed $0\les t\les T$, $Z(t,\cd)$ is $\dbF$-progressively measurable, and

\item equation \rf{bsvie-II} is satisfied in the usual It\^{o} sense for Lebesgue measure almost every $t\in[0,T]$.

\end{enumerate}
Condition (ii) implies that for any $t\in[0,T)$, the random variable $Z(t,s)$ is $\cF_s$-measurable for any $s\in[t,T]$.
In \rf{bsvie-II}, $g$ and $\psi$ are called the {\it generator} and the {\it free term}, respectively.
Let us point out that in this paper, we only study the BSVIEs with $Y(\cd)$ being one-dimensional.
The case that $Y(\cd)$ being higher dimensional will be significantly different in general, and will be investigated in the near future.
However, the Brownian motion $W(\cd)$ assumed to be one-dimensional is just for convenience of our presentation.

\ms

When $Z(s,t)$ is absent, \rf{bsvie-II} is reduced to the form
\bel{bsvie-I}Y(t)=\psi(t)+\int_t^Tg(t,s,Y(s),Z(t,s))ds-\int_t^TZ(t,s)dW(s),\q t\in[0,T], \ee
which is a natural extension of BSDEs, and is a little more general than \eqref{BSVIE1} since $g$ depends on both $t$ and $s$. BSVIEs of form \rf{bsvie-I}, referred to as {\it Type-I BSVIEs}, were firstly studied by Lin \cite{Lin 2002}, followed by several other researchers: Aman and N'Zi \cite{Aman-N'Zi 2005}, Wang and Zhang \cite{Wang-Zhang 2007}, Djordjevi\'{c} and Jankovi\'{c} \cite{Djordjevic-Jankovic 2013,Djordjevic-Jankovic 2015}, Hu and {\O}ksendal \cite{Hu 2018}.

\ms

BSVIEs of the form \rf{bsvie-II} (containing $Z(s,t)$) were firstly introduced by Yong \cite{Yong 2006,Yong 2008}, motivated by the study of optimal control for forward stochastic Volterra integral equations (FSVIEs, for short). We call \rf{bsvie-II} a {\it Type-II BSVIE} to distinguish it from Type-I BSVIEs.
Type-II BSVIE \rf{bsvie-II} has a remarkable feature that its adapted solution, similarly defined as that for Type-I BSVIEs, might not be unique due to lack of restriction on the term $Z(s,t)$ (with $0\les t\les s\les T$).
Suggested by the natural form of the  adjoint equation in the Pontryagin type maximum principle, Yong \cite{Yong 2008} introduced the notion of {\it adapted M-solutions}: A pair $(Y(\cd),Z(\cd\,,\cd))$ is called an adapted M-solution to \rf{bsvie-II}, if in addition to (i)--(iii) stated above, the following condition is also satisfied:
\bel{M-solution} Y(t)=\dbE[Y(t)]+\int_0^tZ(t,s)dW(s),\q\ae~t\in[0,T],~\as \ee
Under usual Lipschitz conditions, well-posedness was established in \cite{Yong 2008} for the adapted M-solutions to Type-II BSVIEs of form \rf{bsvie-II}.
This important development has triggered extensive research on BSVIEs and their applications. For instance, Anh, Grecksch and Yong \cite{Anh-Grecksch-Yong 2011} investigated BSVIEs in Hilbert spaces; Shi, Wang and Yong \cite{Shi-Wang-Yong 2013} studied well-posedness of BSVIEs containing mean-fields (of the unknowns); Ren \cite{Ren 2010}, Wang and Zhang \cite{Wang-Zhang p} discussed BSVIEs with jumps; Overbeck and R\"oder \cite{Overbeck-Roder p} even developed a theory of path-dependent BSVIEs; Numerical aspect was considered by Bender and Pokalyuk \cite{Bender-Pokalyuk 2013}; relevant optimal control problems were studied by Shi, Wang and Yong \cite{Shi-Wang-Yong 2015},
Agram and {\O}ksendal \cite{Agram-Oksendal 2015}, Wang and Zhang \cite{Wang-Zhang 2017}, and Wang \cite{Wang 2018}; Wang and Yong \cite{Wang-Yong 2015} established various comparison theorems for both adapted solutions and adapted M-solutions to BSVIEs in multi-dimensional Euclidean spaces.

\ms

Recently, inspired by the Four-Step Scheme in the theory of forward-backward stochastic differential equations (FBSDEs, for short) (\cite{Ma-Yong 1999}) and the time-inconsistent stochastic optimal control problems (\cite{Yong 2012}), Wang and Yong \cite{Wang-Yong 2018} established a representation of adapted solutions to Type-I BSVIEs and adapted M-solutions to Type-II BSVIEs in terms of the solution to a system of (non-classical) partial differential equations and the solution to a (forward) stochastic differential equation.

\ms

We point out that in all the above-mentioned works on BSVIEs, the generator $g(t,s,y,z,z')$ of the BSVIE \rf{bsvie-II} satisfies a uniform Lipschitz condition in $(y,z,z')$ so that the generator has a linear growth in $(z,z')$.
However, when the generator $g(s,y,z)$ of BSVIE \eqref{BSVIE1} is given by \eqref{g=z2}, it has a quadratic growth in $z$. Hence, a theory needs to be established for BSVIEs with the generators $g(t,s,y,z,z')$ growing quadratically in $z$, which are called quadratic BSVIEs (QBSVIEs, for short, if the quadratic growth of the generator in $z$ needs to be emphasized). We point out that at the moment, we are not able to handle the case that $z'\mapsto g(t,s,y,z,z')$ is quadratic, and it is also lack of motivation for that case.

\ms

Recall that for BSDE \rf{BSDE2}, when $(y,z)\mapsto g(s,y,z)$ satisfies a uniform Lipschitz condition, with $g(\cd\,,0,0)$ being $L^p$-integrable (with some $p>1$), for any $\cF_T$-measurable $L^p$-integrable random variable $\xi$, it admits a unique adapted solution $(Y(\cd),Z(\cd))$ (\cite{Pardoux-Peng 1990, Ma-Yong 1999, Yong-Zhou 1999}) which could be called a recursive utility process for $\xi$. On the other hand, when $z\mapsto g(s,y,z)$ has an up to quadratic growth, the BSDE \rf{BSDE2} is called a {\it quadratic BSDE} (QBSDE, for short).
In 2000, Kobylanski \cite{Kobylanski 2000} established the well-posedness of QBSDE with $\xi$ being bounded. Since then, some efforts have been made by researchers to relax the assumptions on the generator as well as on the terminal value $\xi$. Among relevant works, we would like to mention Briand and Hu \cite{Briand-Hu 2006,Briand-Hu 2008},
Hu and Tang \cite{Hu-Tang 2016}, Briand and Richou \cite{Briand-Richou 2017}, and Zhang \cite[Chapter 7]{Zhang 2017}. Further, BSDEs with superquadratic growth was investigated by Delbaen, Hu and Bao \cite{Delbaen-Hu-Bao 2011}, where some general negative results concerning the well-posedness can be found. Therefore, one can say that the theory of recursive utility for terminal payoff $\xi$ has reached a pretty mature stage.

\ms

The purpose of this paper is to establish the well-posedness of QBSVIEs under certain conditions. The method introduced by Yong \cite{Yong 2008} and techniques found in Briand--Hu \cite{Briand-Hu 2006, Briand-Hu 2008} will be combined and further developed. In addition, a comparison theorem for adapted solutions of Type-I QBSVIEs will be established. Consequently, equilibrium recursive utility processes and continuous-time equilibrium dynamic risk measures will be investigated. See Yong \cite{Yong 2007} and Wang--Yong \cite{Wang-Yong 2015}, Agram \cite{Agram 2018} for some earlier works. See also Di Persio \cite{Di Persio 2014} for stochastic differential utility, and Kromer--Overbeck \cite{Kromer-Overbeck 2017} for dynamical capital allocation by means of BSVIEs.

\ms

The rest of this paper is organized as follows.
In \autoref{Preliminaries},  we introduce some preliminary notations and definitions, and present some lemmas which are of frequent use in the sequel.
\autoref{I-BSVIE} is devoted to the study of existence and uniqueness of adapted solutions for Type-I QBSVIEs, and \autoref{II-BSVIE} is devoted to the study of existence and uniqueness of adapted M-solutions for Type-II QBSVIE.
A comparison theorem for adapted solutions to Type-I QBSVIEs \rf{bsvie-I} will be established in \autoref{Comparison-thm}, and an application of Type-I BSVIEs to continuous-time equilibrium dynamic risk measures will be presented in \autoref{application}. Some conclusion remarks will be collected in \autoref{remarks}.
Finally, a mathematical justification of the BSVIE model is sketched in the appendix.

\section{Preliminaries}\label{Preliminaries}

For $0\les a<b\les T$, we denote by $\cB([a,b])$ the Borel $\si$-field on $[a,b]$
and define the following sets:
\begin{alignat*}{3}
\D[a,b] &\deq\big\{(t,s)\bigm|a\les t\les s\les b\big\},                  \q& \D^c[a,b] &\deq\big\{(t,s)\bigm|a\les s<t\les b\big\},\\
[a,b]^2 &\deq\big\{(t,s)\bigm|a\les t,s\les b\big\}=\D[a,b]\cup\D^c[a,b], \q& \D^*[a,b] &\deq\cl{\D^c[a,b]}.
\end{alignat*}
Note that $\D^*[a,b]$ is a little different from the complement $\D^c[a,b]$ of $\D[a,b]$ in $[a,b]^2$, since both $\D[a,b]$ and $\D^*[a,b]$ contain the diagonal line segment. In the sequel we shall deal with various spaces of functions and processes, which we collect here first
for the convenience of the reader:
\begin{align*}
L^1(a,b)&\ts=\Big\{h:[a,b]\to\dbR~|~h(\cd)~\hb{is $\cB([a,b])$-measurable, }\int_a^b|h(s)|ds<\i\Big\},\\
L^\i_{\cF_b}(\Om)
&\ts=\Big\{\xi:\Om\to\dbR~|~\xi~\hb{is $\cF_b$-measurable and bounded}\Big\},\\
L^\i_{\cF_b}(a,b)
&\ts=\Big\{\f:[a,b]\times\Om\to\dbR~|~\hb{$\f(\cd)$ is $\cB([a,b])\otimes\cF_b$-measurable and bounded}\Big\},\\
L_\dbF^2(a,b)
&\ts=\Big\{\f:[a,b]\times\Om\to\dbR~|~\f(\cd)~\hb{is $\dbF$-progressively measurable, }\dbE\int_a^b|\f(s)|^2ds<\i \Big\},\\
L^\i_\dbF(a,b)
&\ts=\Big\{\f(\cd)\in L^2_\dbF(a,b)\bigm|\hb{$\f(\cd)$ is bounded}\Big\},\\
L_\dbF^2(\Om;C[a,b])
&\ts=\Big\{\f:[a,b]\times\Om\to\dbR~|~\f(\cd)~\hb{is continuous, $\dbF$-adapted, }\dbE\big[\ds\sup_{a\les s\les b}|\f(s)|^2\big]<\i \Big\},\\
L_\dbF^\i(\Om;C[a,b])
&\ts=\Big\{\f(\cd)\in L^2_\dbF(\Om;C[a,b])\bigm|\ds\sup_{a\les t\les b}|\f(t)|\in L^\i_{\cF_b}(\Om)\Big\},\\
L_{\cF_b}^\i(\Om;C^U[a,b])
&\ts=\Big\{\f(\cd)\in L_{\cF_b}^\i(a,b)\bigm|\hbox{there exists a modulus of continuity $\rho:[0,\i)\to[0,\i)$}\\
&\ts\hp{=\Big\{\ } \hbox{such that}~ |\f(t)-\f(s)|\les \rho(|t-s|),~(t,s)\in[a,b],~\as\Big\},\\
L_\dbF^2(\D[a,b])
&\ts=\Big\{\f:\1n\D[a,b]\1n\times\1n\Om\1n\to\1n\dbR~|~\f(t,\cd)~\hb{is $\dbF$-progressively measurable on $[t,b]$, }\ae~t\1n\in\1n[a,b],\\
&\ts\hp{=\Big\{\ }\qq\qq\qq\qq\qq \dbE\int_a^b\int_t^b|\f(t,s)|^2dsdt<\i \Big\},\\
L_\dbF^2([a,b]^2)
&\ts=\Big\{\f:\1n[a,b]^2\1n\times\1n\Om\1n\to\1n\dbR~|~\f(t,\cd)~\hb{is $\dbF$-progressively measurable on $[a,b]$, }\ae~t\1n\in\1n[a,b],\\
&\ts\hp{=\Big\{\ }\qq\qq\qq\qq\qq\dbE\int_a^b\int_a^b|\f(t,s)|^2dsdt<\i \Big\},\\
\cH^2_\D[a,b] &\ts= L_\dbF^2(a,b)\times L_\dbF^2(\D[a,b]),\qq\cH^2[a,b]\ts= L_\dbF^2(a,b)\times L_\dbF^2([a,b]^2).
\end{align*}
Now, we recall the definitions of adapted solutions and adapted M-solutions for Type-I BSVIE \rf{bsvie-I} and Type-II BSVIE \rf{bsvie-II}, respectively (see \cite{Yong 2008}).
\begin{definition}\rm (i) A pair of processes $(Y(\cd),Z(\cd,\cd))\in\cH^2_\D[0,T]$ is called an {\it adapted solution} of BSVIE \rf{bsvie-I} if \rf{bsvie-I} is satisfied in the usual It\^o sense for Lebesgue measure almost every $t\in[0,T]$.

\ms

(ii) A pair of processes $(Y(\cd),Z(\cd,\cd))\in\cH^2[0,T]$ is called an {\it adapted solution} of BSVIE \rf{bsvie-II}
if \rf{bsvie-II} is satisfied in the usual It\^{o} sense for Lebesgue measure almost every $t\in[0,T]$. Further, it is called an {\it adapted M-solution} of BSVIE \rf{bsvie-II} on $[r,T]$ if, in addition, the following holds:
\bel{M1-solution}
Y(s)=\dbE_r[Y(s)]+\int_r^sZ(s,t)dW(t),\q\ae~ s\in[r,T].
\ee
Here, we recall that $\dbE_r=[\,\cd\,|\,\cF_r]$.

\end{definition}
Let $\cM^2[r,T]$ be the set of all $(y(\cd),z(\cd,\cd))\in\cH^2[r,T]$ satisfying  \rf{M1-solution}. Clearly, $\cM^2[r,T]$ is a closed subspace of $\cH^2[r,T]$.
Further, for any $(y(\cd),z(\cd,\cd))\in\cM^2[r,T]$, we have
$$\dbE|y(s)|^2=\dbE\big|\dbE_r[y(s)]\big|^2+\dbE\int_r^s|z(s,t)|^2
dt\ges\dbE\int_r^s|z(s,t)|^2dt,\q\ae~s\in[r,T].$$
It follows that
\begin{align*}
\|(y(\cd),z(\cd,\cd))\|^2_{\cH^2[r,T]}&\equiv\dbE\[\int_r^T|y(s)|^2ds
+\int_r^T\int_r^T|z(s,t)|^2dtds\]\\
&=\dbE\[\int_r^T|y(s)|^2ds+\int_r^T\int_r^s|z(s,t)|^2dtds+\int_r^T
\int_s^T|z(s,t)|^2dtds\]\\
&\les\dbE\[2\int_r^T|y(s)|^2ds+2\int_r^T\int_s^T|z(s,t)|^2dtds\]\\
&\equiv 2\|(y(\cd),z(\cd,\cd))\|_{\cM^2[r,T]}^2\les 2\|(y(\cd),z(\cd,\cd))\|^2_{\cH^2[r,T]},
\end{align*}
which implies that $\| \cd \|_{\cM^2[r,T]}$ is an equivalent norm of $\|\cd\|_{\cH^2[r,T]}$ on $\cM^2[r,T]$.

\ms

Next, we recall the following definition (see \cite{Kazamari 1994} for relevant details).

\begin{definition} \rm
A uniformly integrable $\dbF$-martingale $M=\{M(t):0\les t\les T\}$ with $M(0)=0$ is called a {\it BMO martingale} on $[0,T]$ if
$$ \|M(\cd)\|^2_{{\rm BMO}(0,T)}\deq\sup_{\tau\in\sT[0,T]}\left\|\dbE_\t\big[|M(T)-M(\t)|^2\big]\right\|_\i<\i,$$
where $\sT[0,T]$ is the set of all $\dbF$-stopping times $\t$ valued in $[0,T]$.
\end{definition}

Sometimes, the norm $\|\cd\|_{{\rm BMO}(0,T)}$ is written as  $\|\cd\|_{{\rm BMO}_\dbP(0,T)}$, indicating the dependence on the probability $\dbP$.

\ms

Next, let $X=\{X_t,\cF_t;0\les t\les T\}$ be a measurable, adapted process satisfying
$$\dbP\[\int_0^T|X(s)|^2ds<\i \]=1.$$
Recall the {\it Dol\'ean-Dade exponential} of $X$:
\bel{Girsanov-E}
\cE\{X\}_t\deq e^{\int_0^tX(s)dW(s)-{1\over2}\int_0^t|X(s)|^2ds},\q t\in[0,T],
\ee
and define a probability measure $\cl\dbP$ on $\cF_T$ by
\bel{Girsanov-P}d\cl\dbP=\cE\{X\}_{\1n_T}d\dbP.
\ee
Then, we have the following lemma which is a combination of the Girsanov's theorem (see Karatzas--Shreve \cite{Karatzas-Shreve 2012} for a proof) and a result found in Kazamaki \cite{Kazamari 1994}.

\begin{lemma}\label{lemma-Girsanov} \sl If $t\mapsto\int_0^tX(s)dW(s)$ is a BMO martingale on $[0,T]$, then $\cE\{X\}_t$ is a uniformly integrable martingale and the process $\cl W=\{\cl W(t),\cF_t\bigm|0\les t\les T\}$ defined by
\bel{lemma-Girsanov-tiW}
\cl W(t)\deq W(t)-\int^t_0 X(s)ds,\q0\les t\les T
\ee
is a standard Brownian motion on $(\Om,\cF_T,\cl\dbP)$.
\end{lemma}

Next, we introduce the following spaces. Let $0\les a< b< c\les T$, and
\begin{align*}
&\cl{\rm BMO}(a,b)=\Big\{\f:[a,b]\times\Om\to\dbR~\big|~\f(\cd)\in L_\dbF^2(a,b), \\
&\q \qq\qq\qq\qq \|\f(\cd)\|^2_{\cl{\rm BMO}(a,b)}\deq\sup_{\t\in\sT[a,b]}\Big\|\dbE_\t\[\int_\t^b|\f(s)|^2ds\,
\]\Big\|_\i<\i \Big\},\\
& \cl{\rm BMO}(\D[a,b])=\Big\{\f:\D[a,b]\times\Om\to\dbR~\big|~\f(\cd,\cd)\in  L_\dbF^2(\D[a,b]), \\
&\q \qq\qq\qq\qq \|\f(\cd,\cd)\|^2_{\cl{\rm BMO}\big(\D[a,b]\big)}\deq\esssup_{t\in[a,b]}\sup_{\t\in\sT[t,b]}\,\Big\|\,\dbE_\t
\[\int_\t^b|\f(t,s)|^2ds\]\Big\|_\i<\i \Big\},\\
& \cl{\rm BMO}\big([a,b]\times[b,c]\big)=\Big\{\f:[a,b]\times[b,c]\times\Om\to\dbR~\big|~\f(\cd,\cd)\in  L_\dbF^2([a,b]\times[b,c]), \\
&\q \qq\qq\qq\qq \|\f(\cd,\cd)\|^2_{\cl{\rm BMO}([a,b]\times[b,c])}\deq\esssup_{t\in[a,b]}\sup_{\t\in\sT[b,c]}\,\Big\|\,\dbE_\t
\[\int_\t^c|\f(t,s)|^2ds\]\Big\|_\i<\i \Big\}.
\end{align*}
We note that for $\f(\cd)\in\cl{\rm BMO}(a,b)$, if we let $\f(s)\equiv 0,~s\in[0,a)$, then $\int_0^s\f(r)dW(r);~0\les s\les b$ is a BMO martingale on $[0,b]$.
Similarly, for $\f(\cd\,,\cd)\in\cl{\rm BMO}(\D[a,b])$, if we let $\f(t,s)\equiv 0,~s\in[0,t)$, then $\int_0^s\f(t,r)dW(r);~0\les s\les b$ is a BMO martingale on $[0,b]$ for almost all $t\in[a,b)$. The situation for $\cl{\rm BMO}\big([a,b]\times[b,c]\big)$ is also similar. The following lemma plays a basic role in our subsequent arguments. We refer the reader to \cite[Theorem 3.3]{Kazamari 1994} for the proof and details.

\begin{lemma}\label{lemma-BMO} \sl
For $K>0$, there are constants $c_1,c_2>0$ depending only on K such that for any BMO martingale $M(\cd)$,
we have for any one-dimensional BMO martingale $N(\cd)$ such that $\|N(\cd)\|_{{\rm BMO}(0,T)}\les K$,
$$
c_1\|M(\cd)\|_{{\rm BMO}_{\dbP}(0,T)}\les\|\cl M(\cd)\|_{{\rm BMO}_{\cl \dbP}(0,T)}\les c_2\|M(\cd)\|_{{\rm BMO}_{\dbP}(0,T)},
$$
where $\cl M(\cd)\deq M(\cd)-\lan M, N\ran(\cd)$ and $d\cl\dbP=\bar\cE\{N(\cd)\}_{_T}d\dbP$.
\end{lemma}

We now consider the following BSDE:
\bel{pre-bsde-1-d}
Y(t)=\xi+\int_t^T f(s,Y(s),Z(s))ds-\int_t^T Z(s)dW(s),\q t\in[0,T].
\ee
Let us introduce the following hypothesis.

\begin{taggedassumption}{(A0)}\label{A0} \rm Let the generator $f:[0,T]\times\dbR\times\dbR\times\Om\to\dbR$ be $\cB([0,T]\times \dbR\times\dbR)\otimes\cF_T$-measurable
such that $s\mapsto f(s,y,z)$ is $\dbF$-progressively measurable for all $(y,z)\in \dbR\times\dbR$. There exist constants $\b$, $\g$, $L$ and a function $h(\cd)\in L^1(0,T)$ such that
\begin{align}
\label{|f|}|f(s,y,z)|\les h(s)+\b|y|+{\g\over 2}|z|^{2},\q (s,y,z)\in[0,T]\times\dbR\times \dbR;\\
\nn|f(s,y_1,z_1)-f(s,y_2,z_2)|\les L|y_1-y_2|+L(1+|z_1|+|z_2|)|z_1-z_2|,\\
\label{|f-f|}\qq\qq\qq\qq\qq\qq (s,y_i,z_i)\in[0,T]\times\dbR\times\dbR,~i=1,2.
\end{align}
\end{taggedassumption}

\begin{lemma}\label{lemma-briand-hu} \sl
Let {\rm\ref{A0}} hold. Then, for any $\xi\in L^\i_{\cF_T}(\Om)$, BSDE \rf{pre-bsde-1-d} admits a unique adapted solution $(Y(\cd),Z(\cd))\in L_\dbF^\i(\Om;$ $C[0,T])\times\cl{\rm BMO}(0,T)$. Moreover,
\bel{lemma-briand-hu-main*}
e^{\g|Y(t)|}\les\dbE_t\[e^{\g e^{\b(T-t)}|\xi|+\g\int_t^T|h(s)|e^{\b(s-t)}ds}\].
\ee
\end{lemma}

\begin{proof} \rm By \cite[Theorem 7.3.3]{Zhang 2017}, BSDE \rf{pre-bsde-1-d} admits a unique adapted solution $(Y(\cd),Z(\cd))\in L^\infty_\dbF(\Om;C[0,T])\times L^2_\dbF(0,T)$. Then, by \cite[Theorem 7.2.1]{Zhang 2017}, we see that the adapted solution $(Y(\cd),$ $Z(\cd))\in L^\infty_\dbF(\Om;C[0,T])\times\cl{\rm BMO}(0,T)$. Further, by \cite[Proposition 1]{Briand-Hu 2008}, we have inequality \rf{lemma-briand-hu-main*}.
\end{proof}


\section{Adapted Solution to Type-I QBSVIE}\label{I-BSVIE}

In this section, we will establish the existence and uniqueness of the adapted solution to Type-I QBSVIE. Keep in mind that we may just use ``BSVIE'', instead of ``Type-I QBSVIE'', for convenience. First, let us look at the following simple example.

\begin{example}\rm
Consider the one-dimensional BSVIE:
\bel{ex-3-1} Y(t)=\psi(t)+\int_t^T{Z(t,s)^2\over 2}ds-\int_t^TZ(t,s)dW(s), \ee
where $\psi(\cd)\in L^\i_{\cF_T}(0,T)$, and $W(\cd)$ is a one-dimensional standard Brownian motion. In order to solve  equation \rf{ex-3-1}, we introduce a family of BSDEs parameterized by $t\in[0,T]$:
\bel{ex-3-2} \eta(t,s)=\psi(t)+\int_s^T {\z(t,r)^2\over 2}dr-\int_s^T\z(t,r)dW(r),\q s\in[t,T]. \ee
By \autoref{lemma-briand-hu}, BSDE \rf{ex-3-2} admits a unique adapted solution
$(\eta(t,\cd),\z(t,\cd))\in L_\dbF^\i(\Om;C[t,T])\times\cl{\rm BMO}(t,T)$.
Let
$$ Y(t)=\eta(t,t) ~\hbox{and} ~Z(t,s)=\zeta(t,s),\q (t,s)\in\D[0,T], $$
then
$$ Y(t)=\psi(t)+\int_t^T {Z(t,s)^2\over 2}ds-\int_t^T Z(t,s)dW(s),\q t\in[0,T], $$
which implies that $(Y(\cd),Z(\cd,\cd))$ is an adapted solution to BSVIE \rf{ex-3-1}.
The uniqueness of the solutions to BSVIE \rf{ex-3-1} can be obtained by the following \autoref{thm-exist-unique-no-y}. Moreover, the first term $Y(\cd)$ of the unique solution to BSVIE \rf{ex-3-1} could be solved explicitly:
\bel{Y-ex}
Y(t)=\ln\{\dbE[e^{\psi(t)}|\cF_t]\},\q t\in[0,T].
\ee
\end{example}

Clearly, from the expression \rf{Y-ex}, we see that as long as
$$\sup_{t\in[0,T]}\dbE\[e^{\psi(t)}\]<\infty,$$
by a usual approximation technique, one could find that BSVIE \rf{ex-3-1} will still have the adapted solution with $Y(\cd)$  given by \rf{Y-ex}. Some general exploration in this direction will be carried out elsewhere.

\ms

From the above example, we see that BSVIE \rf{ex-3-1} can be fully characterized by a family of BSDEs \rf{ex-3-2}. The main reason is that the generator of equation \rf{ex-3-1} is independent of $y$. This suggests us first consider a special case of Type-I QBSVIE \rf{bsvie-I}.

\subsection{A special case}

Consider the following BSVIE:
\bel{bsvie-no-y}
Y(t)=\psi(t)+\int_t^T g(t,s,Z(t,s))ds-\int_t^T Z(t,s)dW(s),
\ee
where the generator $g:\D[0,T]\times\dbR\times\Om\to\dbR$ and the free term $\psi:[0,T]\times\Om\to\dbR$  are given maps.
We adopt the following assumption concerning $g(\cd)$, which is comparable with \ref{A0}.
\begin{taggedassumption}{(A1)}\label{A1} \rm Let the generator $g:\D[0,T]\times \dbR\times \Om\to\dbR$ be $\cB(\D[0,T]\times \dbR)\otimes\cF_T$-measurable such that $s\mapsto g(t,s,z)$ is $\dbF$-progressively measurable on $[t,T]$, for all $(t,z)\in [0,T)\times\dbR$. There exist two constants $\g$, $L$ and a function $h(\cd)\in L^1(0,T;\dbR)$ such that
\begin{align*}
&|g(t,s,z)|\les h(s)+{\g\over 2}|z|^{2},\q (t,s,z)\in\D[0,T]\times \dbR;\\
&|g(t,s,z_1)-g(t,s,z_2)|\les L(1+|z_1|+|z_2|)|z_1-z_2|,\q (t,s,z_i)\in\D[0,T]\times\dbR,~i=1,2.
\end{align*}
\end{taggedassumption}

\noindent Now, we state the following existence and uniqueness result of BSVIE \rf{bsvie-no-y}.

\begin{theorem}\label{thm-exist-unique-no-y} \sl
Let {\rm\ref{A1}} hold. Then for any $\psi(\cd)\in L^\i_{\cF_T}(0,T)$, BSVIE \rf{bsvie-no-y} admits a unique adapted solution $(Y(\cd),Z(\cd,\cd))\in L^\i_\dbF(0,T)\times\cl{\rm BMO}(\D[0,T])$.
\end{theorem}

\begin{proof}
We first show the existence of the adapted solution to BSVIE \rf{bsvie-no-y}.
Consider the following BSDEs parameterized by $t\in[0,T]$:
\bel{eta-zeta-no-y}
\eta(t,s)=\psi(t)+\int_s^T g(t,r,\z(t,r))dr-\int_s^T\z(t,r)dW(r),\q s\in[t,T].
\ee
For almost all $t\in[0,T]$, by \autoref{lemma-briand-hu}, under {\rm\ref{A1}},
BSDE \rf{eta-zeta-no-y} admits a unique adapted solution $(\eta(t,\cd),\z(t,\cd))\in L_\dbF^\i(\Om;C[t,T])\times\cl{\rm BMO}(t,T)$.
Let
$$Y(t)=\eta(t,t),\q Z(t,s)=\z(t,s),\q (t,s)\in\D[0,T],$$
then $(Y(\cd),Z(\cd,\cd))\in L^\i_\dbF(0,T)\times\cl{\rm BMO}(\D[0,T])$ and
$$
Y(t)=\psi(t)+\int_t^Tg(t,s,Z(t,s))ds-\int_t^T Z(t,s)dW(s),\q t\in[0,T],
$$
which implies that $(Y(\cd),Z(\cd,\cd))$ is an adapted solution for BSVIE \rf{bsvie-no-y}.

\ms

The uniqueness is followed from the next theorem.
\end{proof}

Consider the following BSVIEs: For $i=1,2$,
\bel{bsvie-no-y-com}
Y_i(t)=\psi_i(t)+\int_t^T g_i(t,s,Z_i(t,s))ds-\int_t^T Z_i(t,s)dW(s),\q t\in[0,T].
\ee
We have the following comparison theorem.

\begin{theorem}\label{thm-comparison-no-y} \sl
Let $g_1(\cd)$ and $g_2(\cd)$ satisfy {\rm\ref{A1}}, $\psi_1(\cd),\psi_2(\cd)\in L^\i_{\cF_T}(0,T)$. Let $(Y_i(\cd),Z_i(\cd,\cd))\in  L^\i_\dbF(0,T)\times\cl{\rm BMO}(\D[0,T])$ be the adapted solution of corresponding BSVIE \rf{bsvie-no-y-com}. Suppose
\bel{psi-g-no-y-com}
\psi_1(t)\les\psi_2(t),\q g_1(t,s,z)\les g_2(t,s,z),\q \as,~\ae~(t,s,z)\in\D[0,T]\times\dbR,
\ee
then we have
\bel{com-no-y}
Y_1(t)\les Y_2(t),\q\as,~\ae~t\in[0,T].
\ee
In particular, if $g_1(\cd)=g_2(\cd)$ and $\psi_1(\cd)=\psi_2(\cd)$, the comparison implies the uniqueness of adapted solution to BSVIEs \rf{bsvie-no-y}.
\end{theorem}

\begin{proof}
We note that
\begin{align}
\nn Y_1(t)-Y_2(t)&=\psi_1(t)-\psi_2(t)+\int_t^T\left[g_1(t,s,Z_1(t,s))-g_2(t,s,Z_2(t,s))\right]ds\\
\label{com-no-y-y1-y2}        &{\hp =}\qq-\int_{t}^{T}\left[Z_1(t,s)-Z_2(t,s)\right]dW(s).
\end{align}
Define the process $\th(\cd,\cd)$ such that
\begin{align}
&\th(t,s)=0,\q(t,s)\in\D^*[0,T];\\
&|\th(t,s)|\les C(1+|Z_1(t,s)|+|Z_2(t,s)|),\q (t,s)\in\D[0,T];\\
\label{com-no-y-beta}&g_1(t,s,Z_1(t,s))-g_1(t,s,Z_2(t,s))
=\big[Z_1(t,s)-Z_2(t,s)\big]\th(t,s),\q(t,s)\in\D[0,T].
\end{align}
Hereafter, $C>0$ stands for a generic constant which could be different from line to line. Then, for almost all $t\in[0,T]$, $W(t;\cd)$ defined by
\bel{W(t,s)}
W(t;s)\deq W(s)-\int_0^s\th(t,r)dr,\q s\in[0,T]
\ee
is a Brownian motion on $[0,T]$ under the equivalent probability measure $\cl{\dbP}_t$ defined by
$$d\cl{\dbP}_t\deq\cE\{\th(t,\cd)\}_{\1n_T}d\dbP. $$
The corresponding expectation is denoted by $\dbE^{\bar\dbP_t}$. Thus, by \rf{com-no-y-y1-y2} and \rf{W(t,s)}, we have
\begin{align*}
Y_1(t)-Y_2(t)&=\psi_1(t)-\psi_2(t)+\int_t^T\left[g_1(t,s,Z_2(t,s))-g_2(t,s,Z_2(t,s))\right]ds\\
             &\hp{=\ }-\int_t^T\left[Z_1(t,s)-Z_2(t,s)\right]dW(t;s).
\end{align*}
Taking the conditional expectation with respect to $\cl{\dbP}_t$ on the both sides of the above equation and then by \rf{psi-g-no-y-com}, we have
\begin{align*}
Y_1(t)-Y_2(t)&=\dbE^{\bar\dbP_t}_t\[\psi_1(t)-\psi_2(t)+\int_t^T\left[g_1(t,s,Z_2(t,s))
-g_2(t,s,Z_2(t,s))\right]ds\]\les 0,\q\as
\end{align*}
Hence, \rf{com-no-y} follows.
\end{proof}

\begin{remark}
\rm \autoref{thm-exist-unique-no-y} and \autoref{thm-comparison-no-y} are both concerned with the BSVIE \rf{bsvie-no-y},
a very special case of Type-\uppercase\expandafter{\romannumeral1} BSVIE \rf{bsvie-I}, in which, the generator $g(\cd)$ is independent of the variable $y$. This makes the BSVIE \rf{bsvie-no-y} much easier to handle. Even though, \autoref{thm-exist-unique-no-y} and \autoref{thm-comparison-no-y} serve as a crucial bridge to the proof of the results for general Type-I BSVIEs.
\end{remark}
\subsection{The general case}
In this subsection, we will consider the following Type-\uppercase\expandafter{\romannumeral1} BSVIE:
\bel{bsvie-1-d}
Y(t)=\psi(t)+\int_t^Tg(t,s,Y(s),Z(t,s))ds-\int_t^TZ(t,s)dW(s),\qq t\in[0,T].
\ee
We first introduce the following assumption, which is comparable to \ref{A0}.
\begin{taggedassumption}{(A2)}\label{A2}\rm
Let the generator $g:\D[0,T]\times\dbR\times\dbR\times\Om\to\dbR$ be $\cB(\D[0,T]\times\dbR\times\dbR)\otimes\cF_T$-measurable such that $s\mapsto g(t,s,y,z)$ is $\dbF$-progressively measurable on $[t,T]$ for all $(t,y,z)\in [0,T]\times\dbR\times\dbR$. There exist two constants $L$ and $\g$ such that:
\begin{align*}
&
|g(t,s,y,z)|\les L(1+|y|)+{\g\over 2}|z|^2,\q\forall (t,s,y,z)\in\D[0,T]\times\dbR\times\dbR;\\
&|g(t,s,y_1,z_1)-g(t,s,y_2,z_2)|\les L\big\{|y_1-y_2|+(1+|z_1|+|z_2|)|z_1-z_2|\big\},\\
&\qq\qq\qq\qq\qq\qq\q\,~\forall (t,s,y_i,z_i)\in\D[0,T]\times\dbR\times\dbR,~i=1,2.
\end{align*}
\end{taggedassumption}

At the same time, we introduce the following additional assumption which will be used to establish a better regularity for the adapted solutions.

\begin{taggedassumption}{(A3)}\label{A3}\rm
Let $g:[0,T]^2\times\dbR\times\dbR\times\Om\to\dbR$ be measurable such that for every $(t,y,z)\in[0,T]\times\dbR\times\dbR$, $s\mapsto g(t,s,y,z)$ is $\dbF$-progressively measurable. There exists a modulus of continuity $\rho:[0,\i)\to[0,\i)$ (a continuous and monotone increasing function with $\rho(0)=0$) such that
\begin{align*}
|g(t,s,y,z)-g(t',s,y,z)|\les\rho(|t-t'|)(1+|y|+|z|^2),\q \forall~t,t',s\in[0,T],~(y,z)\in\dbR\times\dbR.
\end{align*}
\end{taggedassumption}

Note that in \ref{A3}, the generator $g(t,s,y,z)$ is defined for $(t,s)$ in the square domain $[0,T]^2$ instead of the triangle domain $\D[0,T]$, and the uniform continuity of the map $t\mapsto g(t,s,y,z)$ (uniform for $(s,y,z)$ in any bounded set) is assumed. Now, we state the main result of this subsection.

\begin{theorem}\label{thm-bsvie-1-d-exist-unique} \sl
Let {\rm\ref{A2}} hold. Then for any $\psi(\cd)\in L^\infty_{\cF_T}(0,T)$, BSVIE \rf{bsvie-1-d} admits a unique adapted solution $(Y(\cd),Z(\cd,\cd))\in L^\i_{\dbF}(0,T)\times \cl{\rm BMO}(\D[0,T])$.
\end{theorem}

We will prove \autoref{thm-bsvie-1-d-exist-unique} by means of contraction mapping theorem. For any $(U(\cd),V(\cd,\cd))\in L^\i_{\dbF}(0,T)\times\cl{\rm BMO}(\D[0,T])$, consider the following BSVIE:
\bel{bsvie-1-d-uv}
Y(t)=\psi(t)+\int_t^Tg(t,s,U(s),Z(t,s))ds-\int_t^TZ(t,s)dW(s).
\ee
By \autoref{thm-exist-unique-no-y}, BSVIE \rf{bsvie-1-d-uv} admits a unique adapted solution $(Y(\cd),Z(\cd,\cd))\in L^\i_\dbF(0,T)\times\cl{\rm BMO}$ $(\D[0,T])$. Thus, the map
\bel{1-d-gamma}
\G(U(\cd),V(\cd,\cd))\deq(Y(\cd),Z(\cd,\cd)),\q(U(\cd),V(\cd,\cd))\in L^\i_{\dbF}(0,T)\times\cl{\rm BMO}(\D[0,T])
\ee
is well-defined. In order to prove \autoref{thm-bsvie-1-d-exist-unique}, we present the following lemma.

\begin{lemma}\label{le-l-d-Gamma-e-b} \sl
Let {\rm\ref{A2}} hold and $\e\in(0,{1\over2L}]$. Then for any $\psi(\cd)\in L^\i_{\cF_T}(0,T)$, the map $\G(\cd,\cd)$ defined by \rf{1-d-gamma} satisfies the following:
\bel{le-1-d-Gamma-e-b-main}
\G(\cB_\e)\subseteq\cB_\e,
\ee
where $\cB_\e$ is defined by the following:
\bel{1-d-b-e}\ba{ll}
\ns\ds\cB_\e\deq\Big\{(U(\cd),V(\cd,\cd))\in L^\i_{\dbF}(T-\e,T)\times
\cl{\rm BMO}(\D[T-\e,T])\bigm|\\
\ns\ds\qq\qq\qq\qq\|U(\cd)\|_{L^\i_\dbF(T-\e,T)}\les2\|\psi(\cd)\|_\i+1,
\q\|V(\cd\,,\cd)\|^2_{\cl{\rm BMO}(\D[T-\e,T])}\les A\Big\},
\ea\ee
with
$$A={2\over\g^2}e^{\g\|\psi(\cd)\|_\i}+{1\over\g}e^{2(\g+1)
\|\psi(\cd)\|_\i+\g+2}.$$

\end{lemma}

\begin{proof}

For any $(U(\cd),V(\cd,\cd))\in \cB_\e $, consider a family of BSDEs (parameterized by $t\in[0,T]$):
\bel{le-l-d-Gamma-e-b-1}
\eta(t,s)=\psi(t)+\int_s^Tg(t,r,U(r),\z(t,r))dr-\int_s^T\z(t,r)dW(r),\q s\in[t,T].
\ee
Note that $U(\cd)$ is bounded.
For almost all $t\in[T-\e,T]$, by \autoref{lemma-briand-hu}, the above BSDE admits a unique adapted solution $(\eta(t,\cd),\z(t,\cd))\in L^\i_\dbF(\Om;C[t,T])\times\cl{\rm BMO}(t,T)$. Let
\bel{le-1-d-Gamma-e-b-yz}
Y(t)=\eta(t,t),\q Z(t,s)=\z(t,s),\q (t,s)\in\D[T-\e,T].
\ee
Then by \autoref{thm-exist-unique-no-y}, $(Y(\cd),Z(\cd,\cd))\in L^\i_{\dbF}(0,T)\times\cl{\rm BMO}(\D[0,T])$ is the unique adapted solution to BSVIE \rf{bsvie-1-d-uv}. The rest of the proof is divided into two steps.

\ss

{\bf Step 1:} {\it Estimate of $\|Y(\cd)\|_\i$. }

\ms

For BSDE \rf{le-l-d-Gamma-e-b-1}, by \ref{A2}, we have
$$|g(t,r,U(r),\z)|\les L\big(1+|U(r)|\big)+{\g\over2}|\z|^2.$$
Thus, note that $\e\in(0,{1\over 2L}]$, by \autoref{lemma-briand-hu} with $h(s)=L(1+|U(s)|)$, $\g=\g$ and $\b=0$, we have
\bel{le-l-d-Gamma-e-b-9}\ba{ll}
\ns\ds e^{\g|\eta(t,s)|}\les\dbE_s\[e^{\g \big(|\psi(t)|+L\int_s^T(1+|U(r)|)dr\big)}\]\les e^{\g\big[\|\psi(\cd)\|_\i+L\e\big(1+\|U(\cd)\|_{L^\i_\dbF(T-\e,T)}\big)\big]}\\
\ns\ds\qq\q~\les e^{\g(2\|\psi(\cd)\|_\i+1)},\q T-\e\les t\les s\les T,\ea\ee
which is equivalent to
\bel{le-l-d-Gamma-e-b-9.1}|\eta(t,s)|\les2\|\psi(\cd)\|_\i+1,\q T-\e\les t\les s\les T.\ee
Consequently, noting $Y(t)=\eta(t,t)$, one has
$$\|Y(\cd)\|_{L^\i_\dbF(T-\e,T)}\les2\|\psi(\cd)\|_\i+1.$$

{\bf Step 2:} {\it Estimate of $\|Z(\cd\,,\cd)\|^2_{\cl{\rm BMO}(\D[T-\e,T])}$.}

\ms

Define
\bel{le-l-d-Gamma-e-b-2}\p(y)\deq\g^{-2}\big(e^{\g|y|}-\g|y|-1\big);\q y\in\dbR. \ee
%
Then, we have
\bel{le-l-d-Gamma-e-b-3}
\p'(y)
=\g^{-1}[e^{\g|y|}-1]\hb{sgn}(y),\q
\p''(y)
=e^{\g|y|},\ee
which leads to $\p''(y)=\g|\p'(y)|+1$. Applying It\^o's formula to $s\mapsto\p(\eta(t,s))$, we have
\begin{align}
\nn                         &\p(\psi(t))-\p(\eta(t,s))\\
\label{le-l-d-Gamma-e-b-4}  &\q=-\int_s^T \p'(\eta(t,r))g(t,r,U(r),\z(t,r))dr+{1\over 2}\int_s^T \p''(\eta(t,r))|\z(t,r)|^2dr\\
\nn                         &\q\hp{=\ } +\int_s^T\p'(\eta(t,r))\z(t,r)dW(r),\q s\in[t,T].
\end{align}
Taking conditional expectation on the both sides of \rf{le-l-d-Gamma-e-b-4} and by {\rm\ref{A2}}, we have
\begin{align*}
&\p(\eta(t,s))+{1\over 2}\dbE_s\[\int_s^T\p''(\eta(t,r))|\z(t,r)|^2dr\]\\
&~\les\p(\|\psi(\cd)\|_\i)+L\dbE_s\[\int_s^T|\p'(\eta(t,r))|\big(1+|U(r)|\big)dr\]
+{\g\over2}\dbE_s\[\int_s^T|\p'(\eta(t,r))|\,|\z(t,r)|^2dr\].
\end{align*}
Combining this with \rf{le-l-d-Gamma-e-b-3}, one obtains
\bel{le-l-d-Gamma-e-b-5}\p(\eta(t,s))+{1\over 2}\dbE_s\[\int_s^T|\z(t,r)|^2dr\]\les\p(\|\psi(\cd)\|_\i)+L\dbE_s\[
\int_s^T|\p'(\eta(t,r))|(1+|U(r)|)dr\].\ee
Then, noting that $\p(\eta(t,s))\ges0$, we simply drop it to get
\begin{align*}
&\dbE_s\[\int_s^T|Z(t,r)|^2dr\]\les2\p(\|\psi(\cd)\|_\i)+2L\dbE_s\[\int_s^T|\p'(\eta(t,r))|(1+|U(r)|)dr\]\\
&~\les{2\over\g^2}e^{\g\|\psi(\cd)\|_\i}+{2L\over\g}\e e^{\g(2\|\psi(\cd)\|_\i+1)}e^{2(\|\psi(\cd)\|_\i+1)}
 \les{2\over\g^2}e^{\g\|\psi(\cd)\|_\i}+{1\over\g}e^{2(\g+1)\|\psi(\cd)\|_\i+\g+2}.
\end{align*}
Hence,
\bel{Z<A}\|Z(\cd\,,\cd)\|^2_{\cl{\rm BMO}(\D[T-\e,T])}\les{2\over\g^2}e^{\g\|\psi(\cd)\|_\i}+{1\over\g}e^{2(\g+1)
\|\psi(\cd)\|_\i+\g+2}=A.\ee
This proves our claim. \end{proof}

The next result is concerned with the local solution of BSVIE \rf{bsvie-1-d}.

\begin{proposition}\label{pro-l-d-Gamma-e-c} \sl
Let {\rm\ref{A2}} hold and the map $\G(\cd\,,\cd)$ be defined by \rf{1-d-gamma}.
Then there is $\e>0$  such that $\Gamma(\cd\,,\cd)$ is a contraction on $\mathcal{B}_\e$, where $\mathcal{B}_\e$ is defined by \rf{1-d-b-e}.
This implies that BSVIE \rf{bsvie-1-d} admits a unique adapted solution on $[T-\e,T]$.
\end{proposition}

\begin{proof}
Let $\e\in(0,{1\over2L}]$. For any $(U(\cd),V(\cd\,,\cd)),(\wt U(\cd),\wt V(\cd\,,\cd))\in\cB_\e$, set
\bel{pro-l-d-Gamma-e-c-1}
(Y(\cd),Z(\cd,\cd))=\G(U(\cd),V(\cd\,,\cd))\q\hbox{and}\q(\wt Y (\cd),\wt Z(\cd))=\G(\wt U(\cd),\wt V(\cd\,,\cd));
\ee
that is,
\begin{align}
\label{pro-l-d-Gamma-e-c-2} \eta(t,s)&=\psi(t)+\int_s^T g(t,r,U(r),\z(t,r))ds-\int_s^T\z(t,r)dW(r),\\
\label{pro-l-d-Gamma-e-c-3} \wt\eta(t,s)&=\psi(t)+\int_s^T g(t,r,\wt U(r),\wt\z(t,r))dr-\int_s^T\wt\z(t,r)dW(r),
\end{align}
and
\bel{pro-l-d-Gamma-e-c-4}
Y(t)=\eta(t,t),~\wt Y (t)=\wt\eta(t,t),\q Z(t,r)=\z(t,r),~\wt Z(t,r)=\wt\z(t,r).
\ee
By \autoref{le-l-d-Gamma-e-b}, $(Y(\cd),Z(\cd\,,\cd))$ and $(\wt Y(\cd),
\wt Z(\cd\,,\cd))\in\cB_\e$. By \ref{A2}, for almost all $t\in[T-\e,T]$,  we can define the process $\th(t,\cd)$ in an obvious way such that:
\begin{align}
\label{pro-l-d-Gamma-e-c-5-1} &\th(t,s)=0, \q(t,s)\in[T-\e,T]\times[0,t],\\
\label{pro-l-d-Gamma-e-c-5-2}  &|\th(t,s)|\les L(1+|\z(t,s)|+|\wt\z(t,s)|),\q (t,s)\in\D[T-\e,T],\\
\label{pro-l-d-Gamma-e-c-5-3}    &g(t,s,\wt U(s),\z(t,s))-g(t,s,\wt U(s),\wt\z(t,s))=[\z(t,s)-\wt\z(t,s)]\th(t,s).
\end{align}
Note that $(Y(\cd),\z(\cd\,,\cd)),(\wt Y(\cd),\wt\z(\cd\,,\cd))\in\cB_\e$. Thus, by \rf{pro-l-d-Gamma-e-c-5-1}--\rf{pro-l-d-Gamma-e-c-5-2},
\begin{align}
\nn\|\th(\cd,\cd)\|^2_{\cl{\rm BMO}(\D[T-\e,T])}&\les 3L^2T+3L^2\|\z(\cd,\cd)\|^2_{\cl{\rm BMO}(\D[T-\e,T])}+3L^2\|\wt\z(\cd,\cd)\|^2_{\cl{\rm BMO}(\D[T-\e,T])}]\\
&\les3L^2T+6L^2A.
\end{align}
Thus, for almost all $t\in[T-\e,T]$, $\int_0^s\th(t,r)dW(r); 0\les s\les T$ is a BMO martingale and
\bel{pro-1-d-Gamma-e-c-6}
\left\|\int_0^\cd\th(t,r)dW(r)\right\|^2_{{\rm BMO}(0,T)}\les3L^2T+6L^2A.
\ee
By \autoref{lemma-Girsanov}, $W(t;\cdot)$ defined by
\bel{pro-l-d-Gamma-e-c-7}
 W(t;s)\deq W(s)-\int_0^s\th(t,r)dr,\q s\in[0,T]
\ee
is a Brownian motion on $[0,T]$ under the equivalent probability measure $\cl{\dbP}_t$, which is defined by
\bel{pro-l-d-Gamma-e-c-8}
d\cl{\dbP}_t\deq\cE\{\th(t,\cd)\}_{\1n_T}d\dbP.
\ee
Denote the expectation in $\bar\dbP_t$ by $\dbE^{\bar\dbP_t}$.
Combining \rf{pro-l-d-Gamma-e-c-2}, \rf{pro-l-d-Gamma-e-c-3}, and \rf{pro-l-d-Gamma-e-c-5-3}--\rf{pro-l-d-Gamma-e-c-7},
we have
\begin{align}
\nn &\eta(t,s)-\wt\eta(t,s)+\int_s^T[\z(t,r)-\wt\z(t,r)]d W(t;r)\\
\label{pro-l-d-Gamma-e-c-9}&~=\int_s^T\left[g(t,r,U(r),\z(t,r))-g(t,r,\wt U(r),\z(t,r))\right]dr.
\end{align}
Taking square and then taking conditional expectation with respect to
$\bar\dbP_t$ on the both sides of the above equation, we have (noting
$T-\e\les t\les s\les T$)
\begin{align}
\nn                         &|\eta(t,s)-\wt\eta(t,s)|^2+\dbE^{\bar\dbP_t}_s\[\int_s^T
|\z(t,r)-\wt\z(t,r)|^2dr\]\\
\nn                         &~=\dbE^{\bar\dbP_t}_s\Big\{\[\int_s^T\(g(t,r,U(r),\z(t,r))-g(t,r,\wt U(r),\z(t,r))\)dr\]^2\Big\}\\
\label{pro-l-d-Gamma-e-c-10}&~\les\dbE^{\bar\dbP_t}_s\Big\{\[\int_s^T
\(L|U(r)-\wt U(r)|\)dr\]^2\Big\}\\
\nn                         &~\les L^2(T-t)^2\|U(\cd)-\wt U(\cd)\|^2_{L^\i_\dbF(T-\e,T)}\les L^2\e ^2\|U(\cd)-\wt U(\cd)\|^2_{L^\i_\dbF(T-\e,T)}.
\end{align}
Let $s=t$, by \rf{pro-l-d-Gamma-e-c-4} and \rf{pro-l-d-Gamma-e-c-10}, we have
\bel{pro-l-d-Gamma-e-c-11}
\|Y(\cd)-\wt Y(\cd)\|^2_{L^\i_\dbF(T-\e,T)}\les L^2\e^2\|U(\cd)-\wt U(\cd)\|^2_{L^\i_\dbF(T-\e,T)}.
\ee
Also, by \rf{pro-l-d-Gamma-e-c-4}, \rf{pro-l-d-Gamma-e-c-10},  \rf{pro-1-d-Gamma-e-c-6}, and \autoref{lemma-BMO}, there is a constant $C$ (which is depending on $\|\psi(\cd)\|_\i$ and is independent of $t$) such that
\begin{align}
\nn &\sup_{s\in[t,T]}\dbE_s\[\int_s^T|Z(t,r)-\ti Z(t,r)|^2dr\]= \sup_{s\in[t,T]}\dbE_s\[\int_s^T|\z(t,r)-\wt \z(t,r)|^2dr\]\\
\label{pro-l-d-Gamma-e-c-12}&\q\les C\sup_{s\in[t,T]}\dbE^{\bar\dbP_t}_s\[\int_s^T|\z(t,r)-\ti \z(t,r)|^2dr\]\les CL^2\e^2\|U(\cd)-\wt U(\cd)\|^2_{L^\i_\dbF(T-\e,T)}.
\end{align}
Thus,
\bel{pro-l-d-Gamma-e-c-13}
\|Z(\cd,\cd)-\wt Z(\cd,\cd)\|^2_{\cl{\rm BMO}(\D[T-\e,T])}\les CL^2\e^2\|U(\cd)-\wt U(\cd)\|^2_{L^\i_\dbF(T-\e,T)}.
\ee
Combining \rf{pro-l-d-Gamma-e-c-11}--\rf{pro-l-d-Gamma-e-c-13}, we see that for some small $\e>0$, the map $\G(\cd\,,\cd)$ is a contraction on the set $\cB_\e$. Hence, BSVIE \rf{bsvie-1-d} admits a unique adapted solution on $[T-\e,T]$.
\end{proof}

Let us make some comments on the above local existence of the unique adapted solution.

\vskip-1cm

\setlength{\unitlength}{.01in}
~~~~~~~~~~~~~~~~~~~~~~~~~~~~~~~~~~~~~~~~~~~~~~~\begin{picture}(230,240)
\put(0,0){\vector(1,0){170}}
\put(0,0){\vector(0,1){170}}
\put(110,0){\line(0,1){150}}
\put(150,0){\line(0,1){150}}
\put(0,110){\line(1,0){150}}
\put(0,150){\line(1,0){150}}
\thicklines
\put(0,0){\color{red}\line(1,1){150}}
\put(122,135){\makebox(0,0){$\textcircled{\small 1}$}}
\put(55,130){\makebox(0,0){$\textcircled{\small 2}$}}
\put(-10,150){\makebox(0,0)[b]{$\scriptstyle T$}}
\put(150,-12){\makebox(0,0)[b]{$\scriptstyle T$}}
\put(-15,105){\makebox(0,0)[b]{$\scriptstyle T-\e$}}
\put(105,-12){\makebox(0,0)[b]{$\scriptstyle T-\e$}}
\put(180,-5){\makebox{$t$}}
\put(0,180){\makebox{$s$}}
\put(35,80){\makebox(0,0){$\scriptstyle\D[0,T-\e]$}}
\put(75,35){\makebox(0,0){$\scriptstyle\D^*[0,T-\e]$}}
\end{picture}

\bs

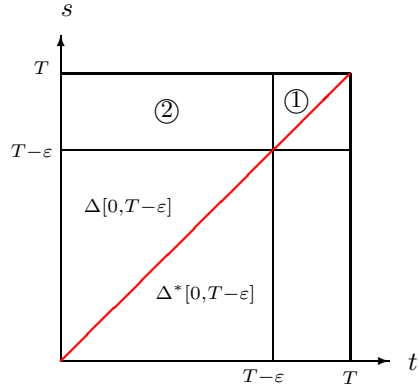
\centerline{(Figure 1)}

\bs

\no We have seen that $(Y(s),Z(t,s))$ is defined for $(t,s)\in\D[T-\e,T]$, the region marked $\textcircled{\small 1}$ in the above figure. Now, for any $t\in[0,T-\e]$, we can rewrite our Type-I BSVIE as follows:
\bel{BSVIE(0,T-e)} Y(t)=\psi^{T-\e}(t)+\int_t^{T-\e}g(t,s,Y(s),Z(t,s))ds-\int_t^{T-\e}
Z(t,s)dW(s),\q t\in[0,T-\e],\ee
where
\bel{psi(T-e)}\psi^{T-\e}(t)=\psi(t)+\int_{T-\e}^Tg(t,s,Y(s),Z(t,s))ds
-\int_{T-\e}^TZ(t,s)dW(s),\q t\in[0,T-\e].\ee
If $\psi^{T-\e}(\cd)\in L^\i_{\cF_{T-\e}}(0,T-\e)$, then \rf{BSVIE(0,T-e)} is a BSVIE on $[0,T-\e]$. However, unlike BSDEs, having $(Y(s),Z(t,s))$ defined on $\D[T-\e,T]$, $\psi^{T-\e}(t); t\in[0,T-\e]$ has still not been defined yet. Since, on the right-hand side of \rf{psi(T-e)}, although $Y(s)$ with $s\in[T-\e,T]$ has already been determined, $Z(t,s)$ has not been defined for $(t,s)\in[0,T-\e]\times[T-\e,T]$, the region marked $\textcircled{\small 2}$ in the above figure, which is needed to define $\psi^{T-\e}(t)$. Moreover, we need that $\psi^{T-\e}(t)$ is $\cF_{T-\e}$-measurable (not just $\cF_T$-measurable). Hence, \rf{psi(T-e)} is actually a {\it stochastic Fredholm integral equation} (SFIE, for short) to be solved to determine $\psi^{T-\e}(t);t\in[0,T-\e]$.

\ms

Now, we are at the position to prove \autoref{thm-bsvie-1-d-exist-unique}.

\ms

\it \textbf{Proof of \autoref{thm-bsvie-1-d-exist-unique}}. \rm  The proof will be divided into three steps.

\ss

{\bf Step 1:} {\it Estimate of $|Y(\cdot)|^2$.}

\ms

For given $\psi(\cd)\in L^\infty_{\cF_T}(0,T)$, we can find a constant $\wt C>0$ such that $\|\psi(\cd)\|_\i^2\les\wt C$ and (by \ref{A2})
\bel{thm-bsvie-1-d-exist-unique-1}
|2xg(t,s,y,0)|\les\wt C+\wt C|x|^2+\wt C|y|^2,\q\forall (t,s,x,y)\in\D[0,T]\times\dbR\times\dbR.
\ee
Let us consider the following (integral form of) ordinary differential equation:
\bel{thm-bsvie-1-d-exist-unique-2}
\a(t)=\wt C+\int_t^T\wt C\a(s)ds+\int_t^T\wt C[\a(s)+1]ds,\q t\in[0,T].
\ee
It is easy to see that the unique solution to the above ordinary differential equation is given by
$$\a(t)=\(\wt C+{1\over2}\)e^{2\wt C(T-t)}-{1\over2},\q t\in[0,T],$$
which is a (continuous) decreasing function. Thus,
$$\|\psi(\cd)\|_\i^2\les\wt C=\a(T)\les\a(0).$$
By \autoref{pro-l-d-Gamma-e-c}, there exists an $\e>0$ (depending on $\|\psi(\cd)\|_\i$) such that $\G(\cd\,,\cd)$ defined by \rf{1-d-gamma} is a contraction on $\cB_\e$. Therefore, a Picard iteration sequence converges to the unique adapted solution $(Y(\cd),Z(\cd,\cd))$ of the BSVIE on $[T-\e,T]$. Namely, if we define:
\bel{thm-bsvie-1-d-exist-unique-4}
 \left\{\begin{aligned}
            (Y^0(\cd),Z^0(\cd,\cd))&=0,\\
            (Y^{k+1}(\cd),Z^{k+1}(\cd,\cd))&=\G(Y^k(\cd),Z^k(\cd,\cd)),\q k\ges0;
\end{aligned}\right.\ee
that is,
\begin{align*}
&(Y^0(\cd),Z^0(\cd,\cd))=0,\\
&\eta^{k+1}(t,s)=\psi(t)+\int_s^T g(t,r,Y^k(r),\z^{k+1}(t,r))dr-\int_s^T\z^{k+1}(t,r)dW(r),\\
&Y^{k+1}(t)=\eta^{k+1}(t,t),\q Z^{k+1}(t,s)=\z^{k+1}(t,s),\q (t,s)\in\D[T-\e,T],
\end{align*}
then
\bel{thm-bsvie-1-d-exist-unique-5}
\lim_{k\to\i}\|(Y^k(\cd),Z^k(\cd,\cd))-(Y(\cd),Z(\cd,\cd))\|_{
L^\i_\dbF(T-\e,T)\times\cl{\rm BMO}(\D[T-\e,T])}=0.
\ee
Next, for almost all $t\in[T-\e,T]$, similar to \rf{pro-l-d-Gamma-e-c-5-2}, \rf{pro-l-d-Gamma-e-c-5-3}, \rf{pro-l-d-Gamma-e-c-7}, and \rf{pro-l-d-Gamma-e-c-8},
there exists a process $\th^{k+1}(t,\cd)$ such that
\bel{thm-bsvie-1-d-exist-unique-6}
g(t,r,Y^k(r),\z^{k+1}(t,r))-g(t,r,Y^k(r),0)=\z^{k+1}(t,r)\th^{k+1}(t,r),
\ee
and
\bel{thm-bsvie-1-d-exist-unique-7}
W^{k+1}(t;s)\deq W(s)-\int_0^s\th^{k+1}(t,r)dr,\q s\in[0,T]
\ee
is a Brownian motion on $[0,T]$ under the corresponding equivalent probability measure $\dbP^{k+1}_t$ defined by
$$\dbP_t^{k+1}=\cE\{\th^{k+1}(t,\cd)\}_{\1n_T}d\dbP.$$
For simplicity, we denote $\dbP^{k+1}_t$ by  $\dbP^{k+1}$ here, suppressing the subscript $t$. The corresponding expectation is denoted by $\dbE^{k+1}$.
It follows that
\begin{align}
                         \nn\eta^{k+1}(t,s)&=\psi(t)+\int_s^Tg(t,r,Y^k(r),\z^{k+1}(t,r))dr-\int_s^T\z^{k+1}(t,r)dW(r),\\
 \label{thm-bsvie-1-d-exist-unique-8}      &=\psi(t)+\int_s^Tg(t,r,Y^k(r),0)dr-\int_s^T\z^{k+1}(t,r)dW^{k+1}(t;r).
\end{align}
Applying the It\^o formula to the map $s\mapsto|\eta^{k+1}(t,s)|^{2}$ and taking conditional expectation  $\dbE^{k+1}_\t=\dbE^{k+1}[\,\cd\,|\,\cF_\t]$ for any $\t\in[T-\e,s]$, by \rf{thm-bsvie-1-d-exist-unique-1}, we have
\begin{align}
\nn& \dbE_\t^{k+1}\[|\eta^{k+1}(t,s)|^2\]+\dbE_\t^{k+1}\[\int_s^T|\z^{k+1}
(t,r)|^2dr\]\\
\label{thm-bsvie-1-d-exist-unique-9}&~=\dbE_\t^{k+1}\[|\psi(t)|^2\]
+\dbE_\t^{k+1}\[\int_s^T2\eta^{k+1}(t,r)g(t,r,Y^k(r),0)dr\]\\
\nn&~\les\wt C+\wt C\int_s^T\dbE_\t^{k+1}\[|\eta^{k+1}(t,r)|^2\]dr+\wt C\int_s^T\Big\{\dbE_\t^{k+1}\[|Y^{k}(r)|^2\]+1\Big\}dr.
\end{align}
We now prove the following inequality by induction:
\bel{thm-bsvie-1-d-exist-unique-10}
|Y^k(t)|^2\les\a(t),\q t\in[T-\e,T],\q\hbox{for any}~k\ges0.
\ee
In fact, by \rf{thm-bsvie-1-d-exist-unique-4}, it is obvious to see $|Y^0(t)|^2=0\les\a(t)$.
Suppose $|Y^k(t)|^2\les\a(t)$ for any $t\in[T-\e,T]$, then
\bel{thm-bsvie-1-d-exist-unique-11}
\dbE_\t^{k+1}\[|\eta^{k+1}(t,s)|^2\]\les\wt C+\wt C\int_s^T\dbE_\t^{k+1}\[|\eta^{k+1}(t,r)|^2\]dr+\wt C\int_s^T[\a(r)+1]dr.
\ee
In light of \rf{thm-bsvie-1-d-exist-unique-2}, by the comparison theorem of ordinary differential equations, we have
\bel{thm-bsvie-1-d-exist-unique-12}
\dbE_\t^{k+1}\[|\eta^{k+1}(t,s)|^2\]\les\a(s).
\ee
Let $\t=s$ and $s=t$, we have
\bel{thm-bsvie-1-d-exist-unique-13}
|Y^{k+1}(t)|^2\les\a(t),\q t\in[T-\e,T].
\ee
Thus, by induction, \rf{thm-bsvie-1-d-exist-unique-10} holds. Then by \rf{thm-bsvie-1-d-exist-unique-5},  we have
\bel{thm-bsvie-1-d-exist-unique-14}
|Y(t)|^{2}\les\a(t),\q t\in[T-\e,T].
\ee

\ss

{\bf Step 2:} {\it A related stochastic Fredholm integral equation is solvable on $[0,T-\e]$.}

\ms

We now solve SFIE \rf{psi(T-e)} on $[0,T-\e]$. Let us introduce a family of BSDEs parameterized by $t\in[0,T-\e]$:
\bel{thm-bsvie-1-d-exist-unique-16}
\eta(t,s)=\psi(t)+\int_s^T g(t,r,Y(r),\z(t,r))dr-\int_s^T\z(t,r)dW(r),\q s\in[T-\e,T].
\ee
By \autoref{lemma-briand-hu}, the above BSDE admits a unique adapted solution $(\eta(t,\cd),\z(t,\cd))$ on $[T-\e,T]$.
Note that \rf{thm-bsvie-1-d-exist-unique-14}, similar to \rf{thm-bsvie-1-d-exist-unique-12}, we have
\bel{thm-bsvie-1-d-exist-unique-171}
|\eta(t,s)|^2\les\a(s),\q s\in[T-\e,T].
\ee
Similar to \rf{Z<A}, we have
\bel{thm-bsvie-1-d-exist-unique-172}
\esssup_{t\in[0,T-\e]}\|\z(t,\cd)\|^2_{\cl{\rm BMO}([T-\e,T])}<\i.
\ee
Let $\psi^{T-\e}(t)=\eta(t,T-\e)$ and $Z(t,s)=\z(t,s)$,
we have $(\psi^{T-\e}(\cd),Z(\cd,\cd))\in L^\i_{\cF_{T-\e}}(0,T-\e)\times\cl{\rm BMO}([0,T-\e]\times[T-\e,T])$
and $(\psi^{T-\e}(\cd),Z(\cd,\cd))$ is a solution to SFIE \rf{psi(T-e)}. Moreover, by \rf{thm-bsvie-1-d-exist-unique-171}, we have
\bel{thm-bsvie-1-d-exist-unique-18}
|\psi^{T-\e}(t)|^2=|\eta(t,T-\e)|^2\les\a(T-\e)\les\a(0),\q t\in[0,T-\e].
\ee
Next, we will prove the solution to SFIE \rf{psi(T-e)} is unique.
Let
$$\ba{ll}
\ns\ds(\psi^{T-\e}(\cd),Z(\cd,\cd)),~(\wt\psi^{T-\e}(\cd),\wt Z(\cd,\cd))\in L^\i_{\cF_{T-\e}}(0,T-\e)\times\cl{\rm BMO}([0,T-\e]\times[T-\e,T]).\ea$$
be two solutions to SFIE \rf{psi(T-e)}. Then
\begin{align}
\label{thm-bsvie-1-d-exist-unique-19}\psi^{T-\e}(t)-\wt\psi^{T-\e}(t)&
=\int_{T-\e}^T\[g(t,s,Y(s),Z(t,s))-g(t,s,Y(s),\wt Z(t,s))\]ds\\
\nn&\hp{=\ }-\int_{T-\e}^T\[Z(t,s)-\wt Z(t,s)\]dW(s),\q t\in[0,T-\e].
\end{align}
For almost all $t\in[0,T-\e]$, similar to \rf{pro-l-d-Gamma-e-c-5-2}, \rf{pro-l-d-Gamma-e-c-5-3}, \rf{pro-l-d-Gamma-e-c-7}, and \rf{pro-l-d-Gamma-e-c-8}, there is a process $\wt\th(t,\cd)$ such that:
\bel{thm-bsvie-1-d-exist-unique-20}
g(t,s,Y(s),Z(t,s))-g(t,s,Y(s),\wt Z(t,s))=[Z(t,s)-\wt Z(t,s)]\wt\th(t,s),
\ee
and
\bel{thm-bsvie-1-d-exist-unique-21}
\cl W (t;s)\deq W(s)-\int_0^s\wt\th(t,r)dr,\q s\in[0,T]
\ee
is a Brownian motion on $[0,T]$ under the corresponding equivalent probability measure $\cl\dbP_t$. The corresponding expectation is denoted by $\dbE^{\bar\dbP_t}$. Combining \rf{thm-bsvie-1-d-exist-unique-19}--\rf{thm-bsvie-1-d-exist-unique-21}, we have
\bel{thm-bsvie-1-d-exist-unique-22}
\psi^{T-\e}(t)-\wt\psi^{T-\e}(t)=-\int_{T-\e}^T\[Z(t,s)-\wt Z(t,s)\]d\cl W(t;s),\q t\in[0,T-\e].
\ee
Taking conditional expectation $\dbE^{\bar\dbP_t}_{T-\e}[\,\cd\,]\equiv\dbE^{\bar\dbP_t}[\,\cd\,|\,
\cF_{T-\e}]$ on the both sides of the equation \rf{thm-bsvie-1-d-exist-unique-22}, we have
\bel{thm-bsvie-1-d-exist-unique-23}
\dbE^{\bar\dbP_t}_{T-\e}\[\psi^{T-\e}(t)-\wt\psi^{T-\e}(t)\]
=0, \q t\in[0,T-\e].
\ee
Note that $\psi^{T-\e}(t)$ is $\cF_{T-\e}$-adapted for any $t\in[0,T-\e]$. It follows that
\bel{thm-bsvie-1-d-exist-unique-24}
\psi^{T-\e}(t)=\wt\psi^{T-\e}(t),\q\as,~t\in[0,T-\e].
\ee
By \rf{thm-bsvie-1-d-exist-unique-22}--\rf{thm-bsvie-1-d-exist-unique-24}, we have
\bel{thm-bsvie-1-d-exist-unique-25}
\int_{T-\e}^T\left[Z(t,s)-\wt Z(t,s)\right]d\cl W(t;s)=0,\q t\in[0,T-\e],
\ee
which implies
\bel{thm-bsvie-1-d-exist-unique-26}
Z(t,s)=\wt Z(t,s),\q\as,~(t,s)\in[0,T-\e]\times[T-\e,T].
\ee
Combining \rf{thm-bsvie-1-d-exist-unique-24}--\rf{thm-bsvie-1-d-exist-unique-26}, SFIE \rf{psi(T-e)} admits a unique solution.

\ss

{\bf Step 3:} {\it Complete the proof by induction.}

\ms
Combining Steps 1 and 2, we have uniquely determined
\bel{thm-bsvie-1-d-exist-unique-27}\left\{\begin{aligned}
   Y(t),\q &\q t\in[T-\e,T],\\
   Z(t,s),&\q (t,s)\in \D[T-\e,T]\bigcup\([0,T-\e]\times[T-\e,T]\).
\end{aligned}\right.\ee
Now, we consider BSVIE \rf{BSVIE(0,T-e)} on $[0,T-\e]$.
By \rf{thm-bsvie-1-d-exist-unique-18}, we see that the above procedure can be repeated. We point out that the introduction of $\a(\cd)$ is to uniformly control the terminal state $\psi(T-\e)$, $\psi(T-2\e)$, etc. Then we can use induction to finish the proof of the existence and uniqueness of adapted solution to BSVIE \rf{bsvie-1-d}.$\hfill\qed$

\ms

\begin{remark}\rm When the terminal condition $\psi(\cd)$ is bounded,  the well-posedness of QBSVIE  \rf{bsvie-1-d} is established by \autoref{thm-bsvie-1-d-exist-unique}. If $\psi(\cd)$ is unbounded, the unboundedness of $\psi(\cd)$ will bring some essential difficulties in establishing the solvability of QBSVIE \rf{bsvie-1-d}. At the moment, we are not able to overcome the difficulties. We hope to come back in our future publications.
\end{remark}

We now would like to look at some better regularity for the adapted solution of BSVIEs under the additional condition \ref{A3}.

\begin{theorem}\label{thm-bsvie-exist-unique-c} \sl
Let {\rm\ref{A2}}--{\rm\ref{A3}} hold. Then for any $\psi(\cd)\in L^\i_{\cF_T}(\Om;C^U[0,T])$, BSVIE \rf{bsvie-1-d} admits a unique adapted solution $(Y(\cd),Z(\cd,\cd))\in L_\dbF^\i(\Om;C[0,T])\times\cl{\rm BMO}(\D[0,T])$.
\end{theorem}

\begin{proof} Without loss of generality, let us assume that
$$|\psi(t')-\psi(t)|\les\rho(|t-t^\prime|),\q \forall~t,t'\in[0,T],$$
with the same modulus of continuity $\rho(\cd)$ given in \ref{A3}.

\ms

By \autoref{thm-bsvie-1-d-exist-unique}, BSVIE \rf{bsvie-1-d} admits a unique adapted solution $(Y(\cd),Z(\cd,\cd))\in L_\dbF^\i(0,T)\times \cl{\rm BMO}(\D[0,T])$. We just need to prove that $Y(\cd)\in L_\dbF^\i(\Om;C[0,T])$, i.e., $Y(\cd)$ is continuous. Consider the following family of BSDEs (parameterized by $t\in[0,T]$):
\bel{thm-a-con-1}
\eta(t,s)=\psi(t)+\int_s^Tg(t,r,Y(r),\z(t,r))dr-\int_s^T\z(t,r)dW(r),
\q s\in[0,T].
\ee
By \autoref{lemma-briand-hu}, for any $t\in[0,T]$, BSDE \rf{thm-a-con-1} admits a unique adapted solution $(\eta(t,\cd),\z(t,\cd))\in L_\dbF^\i(\Om;C[0,T])\times\cl{\rm BMO}(0,T)$.
By \autoref{thm-bsvie-1-d-exist-unique}, we have $Y(t)=\eta(t,t)$, $Z(t,s)=\z(t,s)$ for any $(t,s)\in\D[0,T]$.
Now, let $0\les t<t'\les T$. Similar to \rf{pro-l-d-Gamma-e-c-5-2}, \rf{pro-l-d-Gamma-e-c-5-3}, \rf{pro-l-d-Gamma-e-c-7}, and \rf{pro-l-d-Gamma-e-c-8}, there is a process $\th(t,t';\cd)$ such that
\bel{thm-a-con-2}
g(t',s,Y(s),\z(t,s))-g(t',s,Y(s),\z(t',s))=[\z(t,s)-\z(t',s)]\th(t,t';s),
\ee
and
\bel{thm-a-con-3}
 W (t,t^\prime;s)\deq W(s)-\int_0^s\th(t,t^\prime;r)dr,\q s\in[0,T]
\ee
is a Brownian motion on $[0,T]$ under the corresponding equivalent probability measure $\dbP_{t,t'}$.
The corresponding expectation is denoted by $\dbE^{\dbP_{t,t'}}$.
Combining \rf{thm-a-con-1}, \rf{thm-a-con-2}, and \rf{thm-a-con-3}, we have
\begin{align*}
\eta(t,s)-\eta(t^\prime,s)&=\psi(t)-\psi(t^\prime)-\int_s^T[\z(t,r)-\z(t^\prime,r)]dW(t,t^\prime;r)\\
                        &{\hp=\ } +\int_s^T[g(t,r,Y(r),\z(t,r))-g(t^\prime,r,Y(r),\z(t,r))]dr.
\end{align*}
Taking conditional expectation $\dbE_s^{\dbP_{t,t'}}[\,\cd\,]\equiv\dbE_s^{\dbP_{t,t'}}[\,\cd\,|\cF_s]$ on the both sides of the above equation, we have
\begin{align*}
\eta(t,s)-\eta(t',s)&=\dbE_s^{\dbP_{t,t'}}\[\psi(t)-
                        \psi(t')+\int_s^T\(g(t,r,Y(r),\z(t,r))
-g(t',r,Y(r),\z(t,r))\)dr\].
\end{align*}
Combining this with {\rm\ref{A3}}, by \autoref{lemma-BMO}, we have
\begin{align*}
|\eta(t,s)-\eta(t',s)|&\les\dbE_s^{\dbP_{t,t'}}\[|\psi(t)
                        -\psi(t')|+\int_s^T|g(t,r,Y(r),
\z(t,r))-g(t',r,Y(r),\z(t,r))|dr\]\\
                        &\les \rho(|t-t^\prime|)+\rho(|t-t^\prime|)\dbE_s^{\dbP_{t,t'}}\[\int_s^T(1+|Y(s)|+|\z(t,r)|)dr\]\\
                        & \les C (1+\|Y(\cd)\|_{L_{\dbF}^\i(0,T)}) \rho(|t-t^\prime|)+C\rho(|t-t^\prime|)\dbE_s^{\dbP_{t,t'}}\[\int_s^T|\z(t,r)|^2dr\]\\
                         & \les C (1+\|Y(\cd)\|_{L_{\dbF}^\i(0,T)}) \rho(|t-t^\prime|)+C\rho(|t-t^\prime|)\|\z(t,\cd)\|_{\cl{\rm BMO}_{\dbP_{t,t'}}(t,T)}\\
                         & \les C (1+\|Y(\cd)\|_{L_{\dbF}^\i(0,T)}) \rho(|t-t^\prime|)+C\rho(|t-t^\prime|)\|\z(t,\cd)\|_{\cl{\rm BMO}_{\dbP}(t,T)}\\
                         & \les C (1+\|Y(\cd)\|_{L_{\dbF}^\i(0,T)}+\|\z(\cd,\cd)\|_{\cl{\rm BMO}_{\dbP}(\D[0,T])}) \rho(|t-t^\prime|)\\
                         &= C (1+\|Y(\cd)\|_{L_{\dbF}^\i(0,T)}+\|Z(\cd,\cd)\|_{\cl{\rm BMO}_{\dbP}(\D[0,T])}) \rho(|t-t^\prime|),
\end{align*}
where $C>0$ is  a generic constant (which could be different from line to line).
This leads to
$$\lim_{|t-t'|\to0}\[\sup_{s\in[0,T]}|\eta(t,s)-\eta(t',s)|\]=0,~\as $$
On the other hand, since $\eta(t,\cd)\in L_\dbF^\i(\Om;C[0,T])$ for any $t\in[0,T]$, one has
\bel{state}\lim_{|s-s'|\to 0}|\eta(t,s)-\eta(t,s')|=0,\q\forall t\in[0,T],~\as\ee
It follows that $(t,s)\mapsto\eta(t,s)$ is continuous, i.e.,
$$\lim_{(t',s')\to(t,s)}|\eta(t',s')-\eta(t,s)|=0,\q\forall(t,s)\in[0,T]^2,~\as$$
Consequently, $t\mapsto\eta(t,t)=Y(t)$ is continuous.
\end{proof}

\section{Adapted M-solution to Type-II QBSVIE}\label{II-BSVIE}
We now consider the following one-dimensional Type-II QBSVIE:
\bel{bsvie-m} Y(t)=\psi(t)+\int_t^Tg(t,s,Y(s),Z(t,s),Z(s,t))ds-\int_t^T Z(t,s)dW(s),\qq t\in[0,T]. \ee
Since $Z(s,t)$ is presented in the generator $g(\cd)$, we shall consider the adapted M-solution. Let us first introduce the following assumption:
\begin{taggedassumption}{(A4)}\label{A4} \rm
Let the generator $g:\D[0,T]\times\dbR\times\dbR\times\dbR\times \Om\to\dbR$ be $\cB(\D[0,T]\times\dbR\times\dbR\times\dbR)\otimes\cF_T$-measurable
such that $s\mapsto g(t,s,y,z,z')$ is $\dbF$-progressively measurable on $[t,T]$ for all $(t,y,z,z')\in [0,T]\times\dbR\times\dbR\times\dbR$.
There exist two constants $L$ and $\g$ such that:
\begin{align*}
&|g(t,s,y,z,z')|\les L(1+|y|)+{\g\over 2}|z|^2,\q\forall (t,s,y,z,z')\in\D[0,T]\times\dbR\times\dbR\times\dbR;\\
&|g(t,s,y_1,z_1,z_1')-g(t,s,y_2,z_2,,z_2')|\les L\(|y_1-y_2|+(1+|z_1|+|z_2|)|z_1-z_2|+|z_1'-z_2'|\),\\
&\qq\qq\qq\qq\qq\qq\qq\qq\qq~\forall (t,s,y_i,z_i,z_i')\in\D[0,T]\times\dbR\times\dbR\times\dbR,~ i=1,2.
\end{align*}
\end{taggedassumption}

Note that in \ref{A4}, we have assumed that $z'\mapsto g(t,s,y,z,z')$ is bounded. This will allow us to use the results for Type-I QBSVIEs. Therefore, the following result can be regarded as a byproduct of the results for Type-I QBSVIEs from the previous section. The case that allowing $z'\mapsto g(t,s,y,z,z')$ to be unbounded seems to be more difficult and might be treated in our future investigations. Now, here is the main result of this section.

\begin{theorem}\label{thm-bsvie-m-exist-unique} \sl
Let  {\rm\ref{A4}} hold. Then for any $\psi(\cd)\in L^\i_{\cF_T}(0,T)$, Type-II QBSVIE \rf{bsvie-m} admits a unique adapted M-solution $(Y(\cd),Z(\cd,\cd))\in\cM^2[0,T]\bigcap \lt(L^\i_{\dbF}(0,T)\times \cl{\rm BMO}(\D[0,T])\rt)$.
\end{theorem}

\begin{proof}
For any $(y(\cd),z(\cd,\cd))\in \cM^2[0,T]$, consider the following BSVIE:
\bel{bsvie-m-uv}
Y(t)=\psi(t)+\int_t^Tg(t,s,Y(s),Z(t,s),z(s,t))ds-\int_t^TZ(t,s)dW(s),\q t\in[0,T].
\ee
In light of {\rm\ref{A4}}, by \autoref{thm-bsvie-1-d-exist-unique},
BSVIE \rf{bsvie-m-uv} admits a unique adapted solution $(Y(\cd),Z(\cd,\cd))\in  L_\dbF^\i(0,T)\times\cl{\rm BMO}(\D[0,T])$.
Determine $Z(s,t);(t,s)\in\D[0,T]$ by martingale representation theorem, i.e.,
$$Y(s)=\dbE[Y(s)]+\int_0^sZ(s,t)dW(t),\q s\in[0,T].$$
This means that BSVIE \rf{bsvie-m-uv} admits a unique adapted M-solution $(Y(\cd),Z(\cd,\cd))\in\cM^2[0,T]$.
Thus the map
\bel{m-gamma}
\wt\G(y(\cd),z(\cd,\cd))\deq(Y(\cd),Z(\cd,\cd)),\qq (y(\cd),z(\cd,\cd))\in \cM^2(0,T)
\ee
is well-defined.
In order to prove BSVIE \rf{bsvie-m} admits a unique adapted M-solution, we need to prove that $\wt\G(\cd\,,\cd)$ has a fixed point in $\cM^2[0,T]$. The proof is divided into two steps.
\ms

{\bf Step 1.}  There is an $\e>0$ such that $\wt\G(\cd,\cd)$ is a contraction on $\cM^2[T-\e,T]$ and hence BSVIE \rf{bsvie-m} admits a unique adapted M-solution on $[T-\e,T]$.

\ms
For any $(y(\cd),z(\cd,\cd)),(\wt y(\cd),\wt z(\cd,\cd))\in\cM^2[T-\e,T]$, with $\e>0$ undetermined, set
\bel{m-gamma-yz}
(Y(\cd),Z(\cd,\cd))=\wt\G(y(\cd),z(\cd,\cd)),\q(\wt Y(\cd),\wt Z(\cd\,,\cd))=\wt\G(\wt y(\cd),\wt z(\cd,\cd));
\ee
that is, for $t\in[T-\e,T]$,
\begin{align}
\label{bsvie-m-gamma-yz1} Y(t)&=\psi(t)+\int_t^Tg(t,s,Y(s),Z(t,s),z(s,t))ds-\int_t^TZ(t,s)dW(s),\\
\label{bsvie-m-gamma-yz2} \wt Y(t)&=\psi(t)+\int_t^Tg(t,s,\wt Y(s),\wt Z(t,s),\wt z(s,t) )ds-\int_t^T\wt Z(t,s)dW(s),
\end{align}
and
\begin{align}
\label{bsvie-m-gamma-m1}Y(s)&=\dbE[Y(s)|\cF_{T-\e}]+\int_{T-\e}^s Z(s,t)dW(t),~s\in[T-\e,T],\\
\label{bsvie-m-gamma-m2} \wt Y(s)&=\dbE[\wt Y(s)|\cF_{T-\e}]+\int_{T-\e}^s \wt Z(s,t)dW(t),~s\in[T-\e,T].
\end{align}
Similar to \autoref{le-l-d-Gamma-e-b}, noting that $z'\mapsto g(t,s,y,z,z')$ is bounded, there is an $\e>0$ such that $\wt\G(y(\cd),z(\cd\,,\cd))\in\cB_\e$ for any $(y(\cd),z(\cd\,,\cd))\in\cM^2(T-\e,T)$, where $\cB_\e$ is defined by \rf{1-d-b-e}.
Thus, we have
\bel{(Y,Z),(Y,Z)}
(Y(\cd),Z(\cd\,,\cd)),~(\wt Y(\cd),\wt Z(\cd\,,\cd))\in\cB_\e.
\ee
By {\rm\ref{A4}}, for any $t\in[T-\e,T]$, there is a process $\th(t,\cd)$ such that:
\begin{align}
\label{th-m-Gamma-beta1} & \th(t,s)=0, \q t\in[T-\e,T],~s
\in[0,t],\\
\label{th-m-Gamma-beta2} & |\th(t,s)|\les L(1+|Z(t,s)|+|\wt Z (t,s)|),\q(t,s)\in\D[T-\e,T],\\
\nn & g(t,s,\wt Y(s),Z(t,s),\wt z(s,t))-g(t,s,\wt Y(s),\wt Z(t,s),\wt z(s,t))\\
\label{th-m-Gamma-beta3}                    &~=[Z(t,s)-\wt Z(t,s)]\th(t,s),~\q(t,s)\in\D[T-\e,T].
\end{align}
Similar to \rf{pro-1-d-Gamma-e-c-6}, we have
\bel{th-m-beta-bmo}
\|\th(\cd,\cd)\|^2_{\cl{\rm BMO}(\D[T-\e,T])}\les3L^2T+6L^2A.
\ee
For almost all $t\in[T-\e,T]$, by \autoref{lemma-Girsanov}, $W(t;\cdot)$ defined by
\bel{th-m-w}
 W(t;s)\deq W(s)-\int_0^s\th(t,r)dr,\q s\in[0,T]
\ee
 is a Brownian motion on $[0,T]$ under the equivalent probability measure $\cl{\dbP}_t$,
 which is defined by
\bel{th-m-p}
d\cl{\dbP}_t\deq\cE\{\th(t,\cd)\}_{\1n_T}d\dbP.
\ee
The corresponding expectation is denoted by $\dbE^{\bar\dbP_t}$.
Combining \rf{bsvie-m-gamma-yz1}--\rf{bsvie-m-gamma-yz2} and \rf{th-m-Gamma-beta3}--\rf{th-m-w}, we have
\begin{align}
\nn & Y(t)-\wt Y(t)+\int_t^T [Z(t,s)-\wt Z(t,s)]dW(t,s)\\
\label{th-m-yz-tiyz}&~ =\int_t^T\left[g(t,s,Y(s),Z(t,s),z(s,t))-g(t,s,\wt Y(s),Z(t,s),\wt z(s,t))\right]ds.
\end{align}
Taking square and then taking the conditional expectation $\dbE_t^{\bar\dbP_t}[\,\cd\,]=\dbE^{\bar\dbP_t}[\,\cd\,|\,\cF_t]$, we have
\begin{align}
\nn & |Y(t)-\wt Y(t)|^2+\dbE_t^{\bar\dbP_t}\[\int_t^T|Z(t,s)-\wt Z(t,s)|^2ds\]\\
\nn &~=\dbE_t^{\bar\dbP_t}\[\int_t^T\(g(t,s,Y(s),Z(t,s),z(s,t))
-g(t,s,\wt Y(s),Z(t,s),\wt z(s,t))\)ds\]^2\\
\label{th-m-yz-tiyz-cond}&~\les L^2\dbE_t^{\bar\dbP_t}\[\int_t^T\(|Y(s)-\wt Y(s)|+|z(s,t)-\wt z(s,t)|\)ds\]^2.
\end{align}
By $(Y(\cd),Z(\cd,\cd)),(\wt Y(\cd),\wt Z(\cd,\cd))\in\cB_\e$ and \autoref{lemma-BMO}, there is a constant $C>0$ (which is depending on $\|\psi(\cd)\|_\i$ and is independent of $t$) such that
\begin{align}
\nn &|Y(t)-\wt Y(t)|^2+\dbE_t\[\int_t^T|Z(t,s)-\wt Z(t,s)|^2ds\]\\
\nn&~\les C\dbE_t\[\int_t^T\(|Y(s)-\wt Y(s)|+|z(s,t)-\wt z(s,t)|\)ds\]^2\\
\label{th-m-yz-tiyz-contra}&~\les C(T-t)\dbE_t\[\int_t^T\(|Y(s)-\wt Y(s)|^2+|z(s,t)-\wt z(s,t)|^2\)ds\].
\end{align}
Thus, integrating the above on $[T-\e,T]$, we obtain
\begin{align}
\nn &\dbE\int_{T-\e}^T|Y(t)-\wt Y(t)|^2dt+\dbE\int_{T-\e}^T\int_t^T|Z(t,s)-\wt Z(t,s)|^2dsdt\\
\label{th-m-yz-tiyz-in-contra}&~\les C\e\dbE\int_{T-\e}^T\int_t^T\left[|Y(s)-\wt Y(s)|^2+|z(s,t)-\wt z(s,t)|^2\right]dsdt,
\end{align}
with a possible different constant $C>0$. By the variation of constants formula, we obtain
\begin{align}
\nn &\dbE\int_{T-\e}^T|Y(t)-\wt Y(t)|^2dt+\dbE\int_{T-\e}^T\int_t^T|Z(t,s)-\wt Z(t,s)|^2dsdt\\
\label{th-m-yz-tiyz-in-contra1}&~\les C\e\dbE\int_{T-\e}^T\int_t^T |z(s,t)-\wt z(s,t)|^2dsdt\les C\e\dbE\int_{T-\e}^T|y(t)-\wt y(t)|^2dt.
\end{align}
The constant appears above is generic (only depends on the constants $L$, $\g$, $T$, and $\|\psi(\cd)\|_\i$, and is independent of $\e>0$). Therefore, when $\e$ is small enough, $\wt\G(\cd,\cd)$ is a contraction on $\cM^2(T-\e,T)$. Consequently, BSVIE \rf{bsvie-m} admits a unique adapted solution on $[T-\e,T]$.
Further, by \rf{(Y,Z),(Y,Z)}, the unique adapted M-solution $(Y(\cd),Z(\cd,\cd))$ also belongs to $L^\i_{\dbF}(T-\e,T)\times \cl{\rm BMO}(\D[T-\e,T])$.

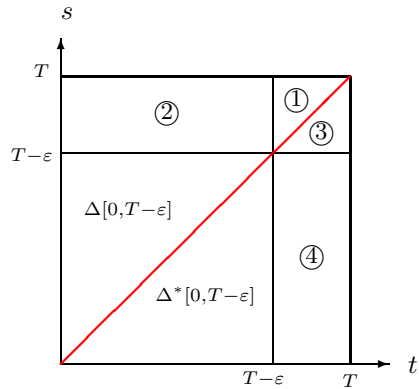
\ms

\vskip-1cm

\setlength{\unitlength}{.01in}
~~~~~~~~~~~~~~~~~~~~~~~~~~~~~~~~~~~~~~~~~~~~~~~\begin{picture}(230,240)
\put(0,0){\vector(1,0){170}}
\put(0,0){\vector(0,1){170}}
\put(110,0){\line(0,1){150}}
\put(150,0){\line(0,1){150}}
\put(0,110){\line(1,0){150}}
\put(0,150){\line(1,0){150}}
\thicklines
\put(0,0){\color{red}\line(1,1){150}}
\put(122,137){\makebox(0,0){$\textcircled{\small 1}$}}
\put(55,130){\makebox(0,0){$\textcircled{\small 2}$}}
\put(135,120){\makebox(0,0){$\textcircled{\small 3}$}}
\put(130,55){\makebox(0,0){$\textcircled{\small 4}$}}
\put(-10,150){\makebox(0,0)[b]{$\scriptstyle T$}}
\put(150,-12){\makebox(0,0)[b]{$\scriptstyle T$}}
\put(-15,105){\makebox(0,0)[b]{$\scriptstyle T-\e$}}
\put(105,-12){\makebox(0,0)[b]{$\scriptstyle T-\e$}}
\put(180,-5){\makebox{$t$}}
\put(0,180){\makebox{$s$}}
\put(35,80){\makebox(0,0){$\scriptstyle\D[0,T-\e]$}}
\put(75,35){\makebox(0,0){$\scriptstyle\D^*[0,T-\e]$}}
\end{picture}

\bs

\centerline{(Figure 2)}

\bs

The above determined $Y(t)$ for $t\in[T-\e,T]$ and determined
$Z(t,s)$ for $(t,s)\in\D[T-\e,T]$ (the region marked $\textcircled{\small 1}$ in the above figure) by using Type-I BSVIEs, and for $(t,s)\in\D^*[T-\e,T]$ (the region marked $\textcircled{\small 3}$ in the above figure) by using martingale representation.

\ms

{\bf Step 2.} BSVIE \rf{bsvie-m} admits a unique adapted M-solution on $[0,T]$.

\ms
By Step 1, BSVIE \rf{bsvie-m} admits a unique solution on $[T-\e,T]$. For almost every $s\in[T-\e,T]$, $\dbE_{T-\e}[Y(s)]\in L^2_{\cF_{T-\e}}(\Om)$, by martingale representation theorem, there is a unique $Z(\cd,\cd)\in L^2(T-\e,T;L^2_\dbF(0,T-\e))$ such that:
\bel{thm-m-z'}
\dbE_{T-\e}[Y(s)]=\dbE[Y(s)|+\int_0^{T-\e}Z(s,t)dW(t),\q s\in[T-\e,T].
\ee
Hence, we have uniquely determined $(Y(t),Z(t,s))$ for $(t,s)\in[T-\e,T]\times[0,T]$ (the region marked $\textcircled{\small 1}$, $\textcircled{\small 3}$ and $\textcircled{\small 4}$) and the following is well-defined:
\bel{thm-m-gs}
g^{T-\e}(t,s,z)=g(t,s,Y(s),z,Z(s,t)),\q (t,s)\in[0,T-\e]\times[T-\e,T].
\ee
Note that $[0,T-\e]\times[T-\e,T]$ is the region marked $\textcircled{\small 2}$ in the above Figure 2. Now, consider the following SFIE:
\bel{thm-m-sfie}
\psi^{T-\e}(t)=\psi(t)+\int_{T-\e}^Tg^{T-\e}(t,s,Z(t,s))ds -\int_{T-\e}^TZ(t,s)dW(s), \q t\in[0,T-\e].
\ee
Similar to the Step 2 of the proof of \autoref{thm-bsvie-1-d-exist-unique}, SFIE \rf{thm-m-sfie} admits a unique solution $(\psi^{T-\e}(\cd),Z(\cd,\cd))$
on $[0,T-\e]\times[T-\e,T]$ and the following estimate holds:
\bel{them-m-psi-l}
|\psi^{T-\e}(t)|^2\les\a(0),\q t\in[0,T-\e],
\ee
where $\a(\cd)$ solves an equation similar to \rf{thm-bsvie-1-d-exist-unique-2}. The above uniquely determined
\bel{thm-m-y-z-psi}\left\{\begin{aligned}
   Y(t), &\q t\in[T-\e,T],\\
   Z(t,s),&\q (t,s)\in \([T-\e,T]\times[0,T]\)\bigcup\([0,T-\e]\times[T-\e,T]\).
\end{aligned}\right.\ee
Now, we consider
\bel{thm-m-bsvie-[T-2e,T-e]}
Y(t)=\psi^{T-\e}(t)+\int_t^{T-\e}g(t,s,Y(s),Z(t,s),Z(s,t))ds-\int_t^{T-\e}Z(t,s)dW(s)
\ee
on $[0,T-\e]$. Since $\psi^{T-\e}(\cd)\in L^\i_{\cF_{T-\e}}(0,T-\e)$, \rf{thm-m-bsvie-[T-2e,T-e]} is a BSVIE on $[0,T-\e]$.
Then the above procedure can be repeated. Since the step-length $\e>0$ can be fixed, we then could use induction to complete the proof.
\end{proof}

\section{A Comparison Theorem for Type-I BSVIEs}\label{Comparison-thm}

Consider the following BSVIEs: For $i=1,2$,\rm
\bel{bsvie-1-d-comparison}
Y^i(t)=\psi^i(t)+\int_t^T g^i(t,s,Y^i(s),Z^i(t,s))ds-\int_t^TZ^i(t,s)dW(s),\q t\in[0,T].
\ee
We assume that the generators $g^i(\cd)$, $i=1,2$ of BSVIEs \rf{bsvie-1-d-comparison} satisfy \ref{A2}. Then by \autoref{thm-bsvie-1-d-exist-unique}, BSVIE \rf{bsvie-1-d-comparison} admits a unique adapted solution
$(Y^i(\cd),Z^i(\cd,\cd))\in L^\i_{\dbF}(0,T)\times\cl{\rm BMO}(\D[0,T])$ for any $\psi^i(\cd)\in L^\i_{\cF_T}(0,T)$.
In order to study the comparison theorem of the solutions to BSVIE \rf{bsvie-1-d-comparison}, we introduce the following BSVIE:
\bel{bsvie-1-d-comparison-bar}
\bar Y(t)=\bar\psi(t)+\int_t^T \bar g(t,s,\bar Y(s),\bar Z(t,s))ds-\int_t^T\bar Z(t,s)dW(s),\q t\in[0,T],
\ee
with the generator $\bar g(\cd)$ also satisfies \ref{A2}.
Further, we adopt the following assumption.

\begin{taggedassumption}{(C)}\label{C}\rm
Let the generator $\bar g:\D[0,T]\times\dbR\times\dbR\times\Om\to\dbR$ satisfy that $y\mapsto \bar{g}(t,s,y,z)$ is nondecreasing for any $(t,s,z)\in\D[0,T]\times\dbR$.
\end{taggedassumption}

We present the comparison theorem for BSVIE \rf{bsvie-1-d-comparison} now.

\begin{theorem}\label{thm-comparison} \sl
Let $g^1(\cd),g^2(\cd)$ and $\bar g(\cd)$ satisfy {\rm\ref{A2}} and let $\bar g(\cd)$ satisfy {\rm\ref{C}}. Suppose
\bel{g<g}g^1(t,s,y,z)\les\bar g(t,s,y,z)\les g^2(t,s,y,z),\q\forall (y,z)\in\dbR\times\dbR,~\as,~\ae~(t,s)\in\D[0,T].\ee
Then for any $\psi^1(\cd),\psi^2(\cd)\in L^\i_{\cF_T}(0,T)$ satisfying
\bel{thm-comparison-c2}
\psi^1(t)\les\psi^2(t),~\as,~\ae~t\in[0,T],
\ee
the corresponding unique adapted solutions $(Y^i(\cd),Z^i(\cd,\cd))$, $i=1,2$ of BSVIEs \rf{bsvie-1-d-comparison} satisfy
\bel{thm-comparison-c3}
Y^1(t)\les Y^2(t),\q\as,~\ae~t\in[0,T].
\ee
If, in addition, the generators $g^1(\cd)$, $g^2(\cd)$ and $\bar g(\cd)$ satisfy {\rm\ref{A3}}, and
\begin{align}
g^1(t,s,y,z)\les\bar g(t,s,y,z)\les g^2(t,s,y,z),\qq\forall (t,y,z)\in[0,T]\times\dbR\times\dbR,~\as,~\ae~s\in[0,T].
\end{align}
Then for any $\psi^1(\cd),\psi^2(\cd)\in L^\i_{\cF_T}(\Om;C^U[0,T])$ satisfying
\bel{thm-comparison-c-c2}
\psi^1(t)\les\psi^2(t),\q t\in[0,T],~\as,
\ee
the corresponding unique adapted solutions $(Y^i(\cd),Z^i(\cd,\cd))$, $i=1,2$ of BSVIEs \rf{bsvie-1-d-comparison} satisfy
\bel{thm-comparison-c-c3}
Y^1(t)\les Y^2(t),\q t\in[0,T],\q\as
\ee
\end{theorem}

\begin{proof} Let $\bar\psi(\cd)\in L^\i_{\cF_T}(0,T)$ such that
\bel{thm-comparison-p1}
\psi^1(t)\les\bar\psi(t)\les\psi^2(t),\q\as,~\ae~t\in[0,T].
\ee
Without loss of generality, let
\bel{thm-comparison-psiL}
\|\psi(\cd)\|_\i\les L,
\ee
where $\psi(\cd)=\psi^1(\cd),\psi^2(\cd),\bar\psi(\cd)$.
By \autoref{thm-bsvie-1-d-exist-unique}, BSVIE \rf{bsvie-1-d-comparison} admits a unique adapted solution
$(Y^1(\cd), Z^1(\cd,\cd))\in L^\i_{\dbF}(0,T)\times\cl{\rm BMO}(\D[0,T])$ for $i=1$. Set $\wt Y_0(\cd)= Y^1(\cd)$ and consider
\bel{thm-comparison-p2}
\wt Y_1(t)=\bar\psi(t)+\int_t^T \bar g(t,s,\wt Y_0(s),\wt Z_1(t,s))ds-\int_t^T \wt Z_1(t,s)dW(s),\q t\in[0,T].
\ee
By \autoref{thm-exist-unique-no-y}, there is a unique adapted solution $(\wt Y_1(\cd),\wt Z_1(\cd,\cd))\in L^\i_{\dbF}(0,T)\times\cl{\rm BMO}(\D[0,T])$ to the above BSVIE. By \rf{g<g}, we have
\bel{thm-comparison-p3}
g^1(t,s,\wt Y_0(s),z)\les\bar g(t,s,\wt Y_0(s),z),\q\forall z\in\dbR,~\as,~\ae~(t,s)\in\D[0,T].
\ee
Combining this and \rf{thm-comparison-p1}, by \autoref{thm-comparison-no-y}, for almost all $t\in[0,T]$, there exists a measurable set $\Om_t^1\subseteq\Om$ satisfying $\dbP(\Om_t^1)=0$ such that
\bel{thm-comparison-p4}
\wt Y_0(t)=Y^1(t)\les\wt Y_1(t),~\om\in\Om\backslash\Om_t^1,~\ae~t\in[0,T].
\ee
Next, we consider the following BSVIE
\bel{thm-comparison-p5}
\wt Y_2(t)=\bar\psi(t)+\int_t^T \bar g(t,s,\wt Y_1(s),\wt Z_2(t,s))ds-\int_t^T \wt Z_2(t,s)dW(s),\q t\in[0,T].
\ee
Let $(\wt Y_2(\cd),\wt Z_2(\cd,\cd))$ be the unique solution to the above equation.
Since $y\mapsto\bar g(t,s,y,z)$ is nondecreasing, by \rf{thm-comparison-p4}, we have
\bel{thm-comparison-p6}
\bar g(t,s,\wt Y_0(s),z)\les \bar g(t,s,\wt Y_1(s),z),\q\forall z\in\dbR,~\as,~\ae~(t,s)\in\D[0,T].
\ee
Similar to the above, for almost everywhere $t\in[0,T]$, there exists a measurable set $\Om_t^2\subseteq\Om$ satisfying $\dbP(\Om_t^2)=0$ such that
\bel{thm-comparison-p7}
\wt Y_1(t)\les \wt Y_2(t),~\om\in\Om\backslash\Om_t^2,~\ae~t\in[0,T].
\ee
By induction, we can construct a sequence $(\wt Y_k(\cd),\wt Z_k(\cd,\cd))$ and $\Om^k_t$ satisfying $\dbP(\Om^k_t)=0$ such that
\bel{thm-comparison-p8}
\wt Y_{k+1}(t)=\bar\psi(t)+\int_t^T \bar g(t,s,\wt Y_k(s),\wt Z_{k+1}(t,s))ds-\int_t^T \wt Z_{k+1}(t,s)dW(s),\q t\in[0,T],
\ee
and
\bel{thm-comparison-p9}
 Y^1(t)=\wt Y_0(t)\les \wt Y_1(t)\les \wt Y_2(t)\les\cds,\q\om\in\Om\setminus\(\bigcup_{k\ges 1}\Om^k_t\),~\ae~t\in[0,T].
\ee
Note that $\dbP[\Om\setminus(\bigcup_{k\geq 1}\Om^k_t)]=0$.
We may assume that
\bel{thm-compariosn-p26}
|\psi(t)|\les\a(0),\q t\in[0,T],
\ee
where $\psi(\cd)=\psi^1(\cd),\psi^2(\cd),\bar\psi(\cd)$ and $\a(\cd)$ solves an ODE of form \rf{thm-bsvie-1-d-exist-unique-2}.
By \autoref{pro-l-d-Gamma-e-c}, there is an $\e>0$ such that $\wt Y_k(\cd)$ is  Cauchy in $L^\i_{\dbF}(T-\e,T)$ and
\bel{thm-comparison-p10}
\lim_{k\to\i}\|\wt Y_k(\cd)-\bar Y(\cd)\|_{L^\i_{\dbF}(T-\e,T)}=0.
\ee
Combining \rf{thm-comparison-p9} and \rf{thm-comparison-p10}, we have
\bel{thm-comparison-p11}
Y^1(t)\les\bar Y(t),\q \as,~\ae~t\in[T-\e,T].
\ee
Next, consider the following SFIEs:
\begin{align}
\label{thm-comparison-p12} \psi^{1,T-\e}(t)&=\psi^1(t)+\int_{T-\e}^Tg^1(t,s,Y^1(s),Z^1(t,s))ds-\int_{T-\e}^T Z^1(t,s)dW(s),\q t\in[0,T-\e];\\
\label{thm-comparison-p13}\bar\psi^{T-\e}(t)&=\bar\psi(t)+\int_{T-\e}^T\bar g(t,s,\bar Y(s),\bar Z(t,s))ds-\int_{T-\e}^T\bar Z(t,s)dW(s), \q t\in[0,T-\e].
\end{align}
Similar to the Step 2 in \autoref{thm-bsvie-1-d-exist-unique},
the above SFIEs \rf{thm-comparison-p12} and \rf{thm-comparison-p13} admit unique solutions
$(\psi^{1,T-\e}(\cd),Z^1(\cd,\cd))$, $(\bar\psi^{T-\e}(\cd),\bar Z(\cd,\cd))\in L^\i_{\cF_{T-\e}}(0,T-\e)\times \cl{\rm BMO}([0,T-\e]\times[T-\e,T])$, respectively.
Similar to \rf{thm-bsvie-1-d-exist-unique-18}, we have
\bel{thm-comparison-p14}
|\psi^{1,T-\e}(t)|\les\a(0),\q|\bar\psi^{T-\e}(t)|\les\a(0),\q t\in[0,T-\e].
\ee
For almost all $t\in[0,T-\e]$, similar to \rf{pro-l-d-Gamma-e-c-5-2}--\rf{pro-l-d-Gamma-e-c-5-3} and \rf{pro-l-d-Gamma-e-c-7}--\rf{pro-l-d-Gamma-e-c-8}, there is a process $\th(t,\cd)$ such that:
\bel{thm-comparison-p15}
g^1(t,s,Y^1(s),Z^1(t,s))-g^1(t,s,Y^1(s),\bar Z(t,s))=\big[Z^1(t,s)-\bar Z(t,s)\big]\th(t,s),
\ee
and
\bel{thm-comparison-p16}
 W (t;s)\deq W(s)-\int_0^s\th(t,r)dr,\q s\in[0,T]
\ee
is a Brownian motion on $[0,T]$ under the corresponding equivalent probability measure $\cl\dbP_t$.
The corresponding expectation is denoted by $\dbE^{\bar\dbP_t}$.
Combining \rf{thm-comparison-p12}--\rf{thm-comparison-p13} and \rf{thm-comparison-p15}--\rf{thm-comparison-p16}, we have
\begin{align}
\nn& \psi^{1,T-\e}(t)-\bar\psi^{T-\e}(t)\\
\nn&~=\psi^1(t)-\bar\psi(t) +\int_{T-\e}^T\big[g^1(t,s,Y^1(s),\bar Z(t,s))-\bar g(t,s,\bar Y(s),\bar Z(t,s))\big]ds\\
\label{thm-comparison-p17}&~\hp{=\ } -\int_{T-\e}^T \big[Z^1(t,s)-\bar Z(t,s)\big]dW(t;s), \q t\in[0,T-\e].
\end{align}
Since $y\mapsto\bar g(t,s,y,z)$ is nondecreasing for any $(t,s,z)\in\D[0,T]\times\dbR$, by \rf{thm-comparison-p11}, we have
\bel{thm-comparison-p25}
\bar g(t,s, Y^1(s),z)\les\bar g(t,s,\bar Y(s),z),\q (t,s,z)\in[0,T]\times[T-\e,T]\times\dbR.
\ee
Taking conditional expectation $\dbE^{\bar\dbP_t}_t[\,\cd\,]\equiv\dbE^{\bar\dbP_t}[\,\cd\,|\,\cd\,]$, on the both sides of \rf{thm-comparison-p17}, by \rf{g<g}, \rf{thm-comparison-p25} and \rf{thm-comparison-p11}, we have
\begin{align}
\nn& \psi^{1,T-\e}(t)-\bar\psi^{T-\e}(t)\\
\nn&~=\dbE_t^{\bar\dbP_t}
\[\psi^1(t)-\bar\psi(t)+\int_{T-\e}^T\big[g^1(t,s,Y^1(s),\bar Z(t,s))-\bar g(t,s,\bar Y(s),\bar Z(t,s))\big]ds\]\\
\label{thm-comparison-p18}&~\les\dbE_t^{\bar\dbP_t}\[\psi^1(t)-\bar
\psi(t)+\int_{T-\e}^T\big[g^1(t,s,Y^1(s),
\bar Z(t,s))-\bar g(t,s,Y^1(s),\bar Z(t,s))\big]ds\]\\
\nn&~\les0,\q t\in[0,T-\e].
\end{align}
Now, we consider the following BSVIEs:
\begin{align}
\label{thm-comparison-p19}                           y^1(t)&=\psi^{1,T-\e}(t)+\int_t^{T-\e}g^1(t,s,y^1(s),z^1(t,s))ds-\int_t^{T-\e}z^1(t,s)dW(s),\q t\in [0,T-\e];\\
\label{thm-comparison-p20}\bar y(t)&=\bar \psi^{T-\e}(t)+\int_t^{T-\e}\bar g(t,s,\bar y(s),\bar z(t,s))ds-\int_t^{T-\e}\bar z(t,s)dW(s),\q t\in[0,T-\e].
\end{align}
By \autoref{thm-bsvie-1-d-exist-unique}, the above equations \rf{thm-comparison-p19}, \rf{thm-comparison-p20} admit unique solutions
$(y^1(\cd),z^1(\cd,\cd))$, $(\bar y(\cd),\bar z(\cd,\cd))\in L^\i_\dbF(0,T-\e)\times\cl{\rm BMO}(\D[0,T-\e])$, respectively.
By the Step 3 in the proof of \autoref{thm-bsvie-1-d-exist-unique}, we have
\begin{align}
y^1(t)=Y^1(t),~z^1(t,s)=Z^1(t,s),\q (t,s)\in\D[0,T-\e]; \\\
\bar y(t)=\bar Y(t),~\bar z(t,s)=\bar Z(t,s),\q (t,s)\in\D[0,T-\e].
\end{align}
Hence, by induction, we have
\bel{thm-comparison-p22}
Y^1(t)\les\bar Y(t),\q\as,~\ae~t\in[0,T].
\ee
Similarly, we can prove that
\bel{thm-comparison-p23}
\bar Y(t)\les Y^2(t),\q\as,~\ae~t\in[0,T].
\ee
Thus, the inequality \rf{thm-comparison-c3} holds.

\ms

Next, by what we have proved,
\bel{thm-comparison-c-p1}
Y^1(t)\les Y^2(t),\q\as,\q t\in[0,T].
\ee
Let $\{t_k\}_{k\ges 1}\subseteq[0,T]$ be all the rational numbers in $[0,T]$. For any fixed $t_k$, by \rf{thm-comparison-c-p1}, there is a $\Om_k\subseteq\Om$  satisfying $\dbP(\Om_k)=0$ such that:
\bel{thm-comparison-c-p2}
Y_1(t_k)\les Y_2(t_k),\qq\om\in\Om\backslash\Om_{t_k}.
\ee
Let $\wt\Om=\bigcup_{k\ges 1}\Om_{t_k}$, then $\dbP(\wt\Om)=0$.
By \rf{thm-comparison-c-p2}, we have
\bel{thm-comparison-c-p3}
Y_1(t)\les Y_2(t),\q t\in\{t_k\}_{k\ges1},~\om\in\Om\backslash\wt\Om.
\ee
By \autoref{thm-bsvie-exist-unique-c}, there is a $\bar\Om\subseteq\Om$ satisfying $\dbP(\bar\Om)=0$ such $Y_i(\cd\,,\om)$, $i=1,2$ are continuous for any $\om\in\Om\backslash\bar\Om$.
For any fixed $\om\in\Om\backslash(\wt\Om\cup\bar\Om)$, by \rf{thm-comparison-c-p3}, we have
\bel{thm-comparison-c-p4}
Y_1(t,\om)\les Y_2(t,\om),\q t\in\{t_k\}_{k\ges 1}.
\ee
Since $Y_i(\cd,\om)$, $i=1,2$ are continuous on $[0,T]$ and $\{t_k\}_{k\ges 1}\subseteq[0,T]$ is dense on $[0,T]$, we have
\bel{thm-comparison-c-p5}
Y_1(t,\om)\les Y_2(t,\om),\q t\in[0,T].
\ee
Note that $\dbP(\Om\backslash(\wt\Om\cup\bar\Om))=0$, we have
\bel{thm-comparison-c-p6}
Y_1(t)\les  Y_2(t),\q t\in[0,T],~\as
\ee
This completes the proof. \end{proof}

\section{Continuous-Time Equilibrium Dynamic Risk Measures}\label{application}

We have seen the so-called equilibrium recursive utility process in the introduction section, which serves as a very important motivation of studying BSVIEs. In this section, we will look another closely related  application of BSVIEs.

\ms

Static risk measures have been studied by many researchers. Among many of them, we mention Artzner--Delbaen--Eber--Heath \cite{Artzner-Delbaen-Eber-Heath 1999}, F\"ollmer--Schied \cite{Follmer-Schied 2002}, and the references cited therein. For discrete-time dynamic risk measures, we mention Riedel \cite{Riedel 2004} and Detlefsen--Scandolo \cite{Detlefsen-Scandolo 2005}, and the references cited therein.

\ms

We now look at continuous-time dynamic risk measures. Any $\xi\in L^\i_{\cF_T}(\Om)$ represents the payoff of certain European type contingent claim at the maturity time $T$. According to El Karoui--Peng--Quenez \cite{El Karoui-Peng-Quenez 1997}, we introduce the following definition.

\begin{definition}\label{dynamic risk} \rm A map $\rho:[0,T]\times L^\i_{\cF_T}(\Om)\to\dbR$ is called a {\it dynamic risk measure} if the following are satisfied:
\begin{itemize}
\item[(i)] (Adaptiveness) For any $\xi\in L^\i_{\cF_T}(\Om)$, $t\mapsto\rho(t;\xi)$ is $\dbF$-adapted;

\item[(ii)] (Monotonicity) For any $\xi,\bar\xi\in L^\i_{\cF_T}(\Om)$ with $\xi\ges\bar\xi$, one has $\rho(t;\xi)\les\rho(t;\bar\xi)$,
      for all $t\in[0,T]$;

\item[(iii)] (Translation Invariant) For any $\xi\in L^\i_{\cF_T}(\Om)$ and $c\in\dbR$, $\rho(t;\xi+c)=\rho(t;\xi)-c$.
\end{itemize}
Further, $\rho$ is said to be {\it convex} if the following holds:
\begin{itemize}
\item[(iv)] (Convexity): $\xi\mapsto\rho(t;\xi)$ is convex;
\end{itemize}
and $\rho$ is said to be {\it coherent} if the following are satisfied:
\begin{itemize}
\item[(v)] (Positive Homogeneity): For any $\xi\in L^\i_{\cF_T}(\Om)$ and $\l\ges0$, $\rho(t;\l\xi)=\l\rho(t;\xi)$;

\item[(vi)] (Subadditivity): For any $\xi,\bar\xi\in L^\i_{\cF_T}(\Om)$, $\rho(t;\xi+\bar\xi)\les\rho(t;\xi)+\rho(t;\bar\xi)$.
\end{itemize}
\end{definition}

Each item in the above definition can be naturally explained. For example, (ii) means that between two gains, the one dominantly larger one has a smaller risk; (vi) means that combining two investments will have smaller risk.
The following is a combination of the results from \cite{El Karoui-Peng-Quenez 1997} and \cite{Kobylanski 2000} (see also \cite{Briand-Hu 2006}, \cite{Briand-Hu 2008}, \cite{Briand-Richou 2017}).

\begin{proposition} \sl Let $g:[0,T]\times\dbR\to\dbR$ be measurable such that $z\mapsto g(t,z)$ is convex and grow at most quadratically. Then for any $\xi\in L^\i_{\cF_T}(\Om)$, the following BSDE:
\bel{BSDE*}Y(t)=-\xi+\int_t^Tg(s,Z(s))ds-\int_t^TZ(s)dW(s),\q t\in[0,T],\ee
admits a unique adapted solution $(Y(\cd),Z(\cd))\equiv (Y(\cd\,;\xi),Z(\cd\,;\xi))$. Let $\rho:[0,T]\times L^\i_{\cF_T}
(\Om)\to\dbR$ be defined by the following:
$$\rho(t,\xi)=Y(t;\xi),\q (t,\xi)\in[0,T]\times L^\i_{\cF_T}(\Om).$$
Then $\rho$ is a dynamic convex risk measure.

\end{proposition}

One of the most interesting examples is the following.
$$Y(t)=-\xi+\int_t^T{1\over2\g}|Z(s)|^2ds-\int_t^TZ(s)dW(s),\q t\in[0,T].$$
The above admits a unique adapted solution $(Y(\cd),Z(\cd))$, and
$$\rho(t,\xi)\equiv Y(t)=\g\ln\dbE\Big[e^{-{\xi\over\g}}\Bigm|\cF_t\Big]\deq e_{\g,t}(\xi),\q t\in[0,T],$$
is called a {\it dynamic entropic risk measure} for $\xi$.

\bs

Now, if we have an $\cF_T$-measurable wealth flow process $\psi(\cd)$ instead of just a terminal payoff $\xi$, then formally, the corresponding dynamic risk should be measured via the following parameterized BSDE:
$$Y(t,r)=-\psi(t)+\int_r^Tg(s,Y(t,s),Z(t,s))ds-\int_r^TZ(t,s)dW(s)), \q(r,t)\in\D[0,T],$$
and the current dynamic risk should be $Y(t;t)$. But, similar to the introduction section, simply taking $r=t$ in the above leads to the following:
$$Y(t,t)=-\psi(t)+\int_t^Tg(s,Y(t,s),Z(t,s))ds-\int_t^TZ(t,s)dW(s)), \q t\in[0,T],$$
which is not a closed form equation for the pair $(Y(t,t),Z(t,s))$ of processes. As we indicated in the introduction, $Y(t,r)$ above has some hidden {\it time-inconsistency} nature. One expects that the dynamic risk measure should be time-consistent. Namely, the value of the risk today (for a process $\psi(\cd)$) should match the one that one expected yesterday. Therefore, it is natural to use BSVIEs to describe/measure the dynamic risk of the process $\psi(\cd)$. We now make this precise.

\ms

We call $\psi(\cd)\in L^\i_{\cF_T}(0,T)$ a position process (a name borrowed from \cite{Riedel 2004}), and $\psi(t)$ could represent the total (nominal) value of certain portfolio process which might be a combination of certain (say, European type) contingent claims (which are mature at time $T$, thus they are usually only $\cF_T$-measurable), some current cash flows (such as dividends to be received, premia to be paid), positions of stocks, mutual funds, and bonds, and so on, at time the current time $t$.
Thus, the position process $\psi(\cd)$ is merely $\cF_T$-measurable (not necessarily $\dbF$-adapted). Now, mimicking Definition \ref{dynamic risk}, we introduce the following.

\begin{definition}\label{dynamic-risk-measures}\rm
A map $\rho:[0,T]\times L^\i_{\cF_T}(0,T)\to  L^\i_{\dbF}(0,T)$ is called an {\it equilibrium dynamic risk measure} if the following hold:

\ms

(i) (Past Independence) For any $\psi_1(\cd),\psi_2(\cd)\in L^\i_{\cF_T}(0,T)$, if
$$\psi_1(s)=\psi_2(s),\q\as,~\ae~s\in[t,T],$$
for some $t\in[0,T)$, then
$$\rho(t;\psi_1(\cd))=\rho(t;\psi_2(\cd)),\q\as$$

(ii) (Monotonicity) For any $\psi_1(\cd),\psi_2(\cd)\in L^\i_{\cF_T}(0,T)$, if
$$\psi_1(s)\les\psi_2(s), \q\as,~\ae~s\in[t,T],$$
for some $t\in[0,T)$, then
$$\rho(s;\psi_1(\cd))\ges\rho(s;\psi_2(\cd)),\q\as,~s\in[t,T].$$

(iii) (Translation Invariance) There exists a deterministic integrable function $r(\cd)$ such that for any $\psi(\cd)\in L^\i_{\cF_T}(0,T)$,
$$\rho(t;\psi(\cd)+c)=\rho(t;\psi(\cd))-ce^{\int_t^Tr(s)ds},\q\as,~t\in[0,T].$$

\no Further, $\rho$ is said to be {\it convex} if the following holds:

\ms

(iv) (Convexity) For any $\psi_1(\cd),\psi_2(\cd)\in L^\i_{\cF_T}(0,T)$ and $\l\in[0,1]$,
$$\rho(t;\l\psi_1(\cd)+(1-\l)\psi_2(\cd))\les\l\rho(t;\psi_1(\cd))
+(1-\l)\rho(t;\psi_2(\cd)),\q\as,~t\in[0,T].$$

\no And $\rho$ is said to be {\it coherent} if the following are satisfied:

\ms

(v) (Positive Homogeneity) For any $\psi(\cd)\in L^\i_{\cF_T}(0,T)$ and $\l>0$,
$$\rho(t;\l\psi(\cd))=\l\rho(t;\psi(\cd)),\q\as,~t\in[0,T].$$

(vi) (Subadditivity) For any $\psi_1(\cd),\psi_2(\cd)\in L^\i_{\cF_T}(0,T)$,
$$\rho(t;\psi_1(\cd)+\psi_2(\cd))\les\rho(t;\psi_1(\cd))
   +\rho(t;\psi_2(\cd)),\q\as,~t\in[0,T].$$
\end{definition}

The word ``equilibrium'' indicates the time-consistency of the risk measure $\rho$ which is some kind of modification of the naive one. Similar situation has happened in the study of time-inconsistent optimal control problems (see \cite{Yong 2012}). The meaning of each item is similar to the static case. In (iii), the function $r(\cd)$ is the riskless interest rate.

\ms

Let us now look at the following Type-I BSVIE:
\bel{appli-absvie}
Y(t)=-\psi(t)+\int_t^Tg(t,s,Y(s),Z(t,s))ds-\int_t^TZ(t,s)dW(s),\q t\in[0,T].
\ee
We have the following result.

\begin{proposition}\label{risk-measure-rho-trans} \sl
Let the generator be given by
$$g(t,s,y,z)\equiv r(s)y+g_0(t,s,z);(t,s,y,z)\in\D[0,T]\times\dbR\times\dbR,$$
satisfying {\rm\ref{A2}}, where $r(\cd)$ is a non-negative deterministic function. Then the following are true:
\begin{enumerate}[\rm(i)]
\item The map $\psi(\cd)\mapsto\rho(t;\psi(\cd))$ is translation invariant.

\item Suppose $z\mapsto g_0(t,s,z)$ is convex, so is $\psi(\cd)\mapsto\rho(t;\psi(\cd))$.

\item Suppose $z\mapsto g_0(t,s,z)$ is positively homogeneous and sub-additive, so is $\psi(\cd)\mapsto\rho(t;\psi(\cd))$.
\end{enumerate}
\end{proposition}

%
%
%
%
%
%

By \autoref{thm-comparison}, the proof of \autoref{risk-measure-rho-trans}
is very similar to \cite[Corollary 3.4, Proposition 3.5]{Yong 2007}, we omit them here. By \autoref{risk-measure-rho-trans}, we can construct a large class of equilibrium dynamic risk measures by choosing suitable generator $g(\cd)$ of BSVIE \rf{appli-absvie}. More precisely, we have the following result.
\begin{theorem}\label{risk-measure-main} \sl
Let the generator $g(t,s,y,z)\equiv r(s)y+g_0(t,s,z);(t,s,y,z)\in\D\times\dbR\times\dbR$ satisfy {\rm\ref{A2}},
where $r(\cd)$ is a non-negative deterministic function and
$z\mapsto g_0(t,s,z)$ is convex, then $\psi(\cd)\mapsto\rho(t;\psi(\cd))$ is an equilibrium dynamic convex risk measure. If $z\mapsto g_0(t,s,z)$ is positively homogeneous and sub-additive, then $\psi(\cd)\mapsto\rho(t;\psi(\cd))$ is an equilibrium dynamic coherent risk measure.
\end{theorem}

From \autoref{risk-measure-rho-trans}, the proof of the above result is obvious. According to the above results, we can have some examples of equilibrium dynamic risk measures by the choices of $g_0(t,s,z)$: If
$$g_0(t,s,z)=\bar g(t,s)|z|,\qq\bar g(t,s)\ges0,$$
then, it is sub-additive and positively homogeneous in $z$. The
corresponding equilibrium dynamic risk measure is coherent. If
$$g_0(t,s,z)=\bar g(t,s)\sqrt{1+|z|^2},\qq\bar g(t,s)\ges0,$$
then, it is convex in $z$. The corresponding equilibrium dynamic risk measure is convex. If
$$g_0(t,s,z)=\bar g(t,s)|z|^2,\qq\bar g(t,s)\ges0,$$
then one has an entropy type equilibrium dynamic risk measure.

\section{Concluding Remarks}\label{remarks}

Recursive utility process (or stochastic differential utility process) and dynamic risk measures for terminal payoff can be described by the adapted solutions to proper BSDEs. For $\cF_T$-measurable position process $\psi(\cd)$, instead of the terminal payoff $\xi$, one could also try to find its recursive utility process and/or dynamic risk. One possibility is again to use BSDEs. However, one immediately finds that the resulting processes (recursive utility or dynamic risk measure) are kind of time-inconsistent nature. Type-I BSVIEs turn out to be a proper tool for describing them. This serves one of major motivations of studying BSVIEs. Recall from \cite{Yong 2006, Yong 2008}, we know that mathematical extension of BSDEs and optimal control of forward stochastic Volterra integral equations are other two motivations. To meet the needs for the equilibrium recursive utility processes and equilibrium dynamic risk measures, we have to allow the generator of the BSVIE to have a quadratic growth in $Z(t,s)$. We have developed a theory of Type-I QBSVIEs, including the well-posedness, regularity and a comparison theorem, etc. in this paper. As a byproduct, we also have obtained the well-posedness of Type-II QBSVIEs. Then a theory of equilibrium recursive utility and equilibrium dynamic risk measures are successfully established with the results of Type-I QBSVIEs.

\bs

\no\bf Acknowledgement. \rm The authors would like to thank two anonymous referees for their suggestive comments which leads to the current version of the paper.

\section{Appendix.}

In this appendix, we will sketch an argument supporting the BSVIE model for the equilibrium recursive utility process/equilibrium dynamic risk measure of a position process $\psi(\cd)$. The idea is adopted from \cite{Yong 2012}. Let $\psi(\cd)$ be a continuous $\cF_T$-measurable process.  Let $\Pi=\{t_k~|~0\les k\les N\}$ be a partition of $[0,T]$ with $0=t_0<t_1<...<t_{N-1}<t_N=T$. The mesh size of $\Pi$ is denoted by $\|\Pi\|\deq\ds\max_{0\les i\les N-1}|t_{i+1}-t_i|$. Let
$$\psi^\Pi(t)=\sum_{k=1}^N\psi_k{\bf1}_{(t_{k-1},t_k]}(t),$$
with
$$\psi_k=\psi(t_k)\in L^2_{\cF_T}(\Om;\dbR),\qq k=1,2,\cds,N.$$
We assume that
$$\lim_{\|\Pi\|\to0}\sup_{t\in[0,T]}\dbE|\psi^\Pi(t)-\psi(t)|^2=0.$$
We first try to specify the time-consistent recursive utility process for $\psi^\Pi(\cd)$, making use of BSDEs. Then let $\|\Pi\|\to0$ to get our BSVIE time-consistent recursive utility process model for $\psi(\cd)$.

\ms

For $\{\psi^\Pi(t)~|~t\in(t_{N-1},t_N]\}=\{\psi_N\}$, its recursive utility at $t\in[t_{N-1},t_N]$ is given by $Y^N(t)$, where $(Y^N(\cd),Z^N(\cd))$ is the adapted solution to the following BSDE:
\bel{BSDE(N)}Y^N(t)=\psi_N+\int_t^Tg(s,Y^N(s),Z^N(s))ds-\int_t^TZ^N(s)dW(s),\qq t\in[t_{N-1},t_N].\ee
Here, $g:[0,T]\times\dbR\times\dbR\to\dbR$ is an aggregator. Next, for $\{\psi^\Pi(t)~|~t\in(t_{N-2},t_N]\}$, the recursive utility at $t\in(t_{N-2},t_{N-1}]$ is denoted by $Y^{N-1}(t)$ and we should have
\bel{BSDE(N-1)}\ba{ll}
\ns\ds Y^{N-1}(t)=\psi_{N-1}+\int_{t_{N-1}}^Tg(s,Y^N(s),Z^{N-1}(s))ds+\int_t^{t_{N-1}}
g(s,Y^{N-1}(s),Z^{N-1}(s))ds\\
\ns\ds\qq\qq\qq-\int_t^TZ^{N-1}(s)dW(s),\qq t\in(t_{N-2},t_{N-1}].\ea\ee
Note that due to the time-consistent requirement, we have to use the already determined $Y^N(\cd)$ in the drift term over $[t_{N-1},T]$. On the other hand, since $\psi_{N-1}$ is still merely $\cF_T$-measurable, \rf{BSDE(N-1)} has to be solved in $[t,T]$ although $t\in(t_{N-2},t_{N-1}]$. Hence, in the martingale term, $Z^{N-1}(\cd)$ has to be free to choose over the entire $[t_{N-2},T]$ and the already determined $Z^N(\cd)$ cannot be forced to use there (on $[t_{N-1},T]$). Whereas, in the drift term over $[t_{N-1},T]$, it seems to be fine to either use already determined $Z^N(\cd)$ or to freely choose $Z^{N-1}(\cd)$, since the time-inconsistent requirement is not required for $Z$ part. However, we use $Z^{N-1}(\cd)$ in the drift, which will enable us to avoid a technical difficulty for BSVIEs later.

\ms

Similarly, the recursive utility on $(t_{N-3},t_{N-2}]$ should be
$$\ba{ll}
\ns\ds Y^{N-2}(t)=\psi_{N-2}+\int_{t_{N-1}}^Tg(s,Y^N(s),Z^{N-2}(s))ds+\int_{t_{N-2}}^{t_{N-1}}
g(s,Y^{N-1}(s),Z^{N-2}(s))ds\\
\ns\ds\qq\qq\qq+\int_t^{t_{N-2}}g(s,Y^{N-2}(s),Z^{N-2}(s))ds-\int_t^TZ^{N-2}(s)dW(s),
\q t\in(t_{N-3},t_{N-2}].\ea$$
This procedure can be continued inductively. In general, we have
\begin{align*}
\ns\ds Y^k(t)&=\psi_k+\sum_{i=k+1}^N\int_{t_{i-1}}^{t_i}g(s,Y^i(s),Z^k(s))ds
+\int_t^{t_k}g(s,Y^k(s),Z^k(s))ds\\
&\qq-\int_t^TZ^k(s)dW(s), \q t\in(t_{k-1},t_k].
\end{align*}
Let us
%
%
 denote
$$Y^\Pi(t)=\sum_{k=1}^NY^k(t){\bf1}_{(t_{k-1},t_k]}(t),\qq Z^\Pi(t,s)=\sum_{k=1}^NZ^k(s){\bf1}_{(t_{k-1},t_k]}(t).$$
Then
$$Y^\Pi(t)=\psi^\Pi(t)+\int_t^Tg(s,Y^\Pi(s),Z^\Pi(t,s))ds-\int_t^TZ^\Pi(t,s)dW(s),\q
t\in[0,T].$$
Let $\|\Pi\|\to0$, by the stability of adapted solutions to BSVIEs (\cite{Yong 2008}), we obtain
\bel{BSVIE*}Y(t)=\psi(t)+\int_t^Tg(s,Y(s),Z(t,s))ds-\int_t^TZ(t,s)dW(s),\q t\in[0,T],\ee
which is the BSVIE that we expected. Moreover, it is found that if $Y(\cd)$ is a  utility process for $\psi(\cd)$, the current utility $Y(t)$ depends on the (realistic) future utilities $Y(r);\,t\les r\les T$, which is the main character of recursive utility process. Finally, we note that if we restrict $Z^{N-1}(\cd)$ on $[t_{N-1},T]$ in \rf{BSDE(N-1)}, etc., then we will end up with the following BSVIE:
$$Y(t)=\psi(t)+\int_t^Tg(s,Y(s),Z(s,s))ds-\int_t^TZ(t,s)dW(s),\q t\in[0,T],$$
which is technically difficult since in general, $s\mapsto Z(s,s)$ is not easy to define.

\ms

Finally, we would like to point out a fact about BSVIEs and BSDEs. Let us first look at the following general BSDE:
\bel{BSDE*}Y(t)=\xi+\int_t^Tg(s,Y(s),Z(s))ds-\int_t^TZ(s)dW(s),\qq t\in[0,T].\ee
Under standard conditions, for any $\xi$ in a proper space, the above BSDE admits a unique solution $(Y(\cd),Z(\cd))\equiv(Y(\cd\,;T,\xi),Z(\cd\,;T,\xi))$. By the uniqueness of adapted solutions of BSDEs, we have
$$Y(t;T,\xi)=Y(t;\t,Y(\t;T,\xi)),\q Z(t;T,\xi)=Z(t;\t,Y(\t;T,\xi)),\qq\forall0\les t<\t\les T.$$
This can be referred to as a (backward) {\it semi-group property} of BSDEs (\cite{Peng 1997}). However, there is no way to talk about the (backward) semi-group property for BSVIEs. To illustrate this point, let us look at the following simple BSVIE:
$$Y(t)=tW(T)-\int_t^TZ(t,s)dW(s),\qq t\in[0,T].$$
We can directly check that the adapted solution is given by
$$Y(t)=tW(t),\q Z(t,s)=t,\qq(t,s)\in\D[0,T].$$
We see that the above $Y(\cd)$ really could not be related to any (backward) semi-group property. The point that we want to make is that time-consistency and semi-group property are irrelevant.


\begin{thebibliography}{90}
\addtolength{\itemsep}{-1.0ex}


\bibitem{Agram 2018} N.~Agram,
\it Dynamic risk measure for BSVIE with jumps and semimartingale issues,
\sl Stoch. Anal. Appl.,
\rm {\bf 37} (2019), 1--16.

\bibitem{Agram-Oksendal 2015} N.~Agram and B.~{\O}ksendal,
\it Malliavin calculus and optimal control of stchastic Volterra equations,
\sl J. Optim. Theory Appl.,
\rm {\bf167} (2015), 1070--1094.

\bibitem{Aman-N'Zi 2005} A.~Aman and M.~N'Zi,
\it Backward stochastic nonlinear Volterra integral equation with local Lipschitz drift,
\sl Probab. Math. Statist.,
\rm {\bf25} (2005), 105--127.

\bibitem{Anh-Grecksch-Yong 2011} V.~V.~Anh, W.~Grecksch, and J.~Yong,
\it Regularity of backward stochastic Volterra integral equations in Hilbert spaces,
\sl Stoch. Anal. Appl.,
\rm {\bf 29} (2011), 146--168.

\bibitem{Artzner-Delbaen-Eber-Heath 1999} P.~Artzner, F.~Delbaen, J.~M. Eber, and D. Heath,
\it Coherent measures of risk,
\sl Math. Finance,
\rm {\bf 9} (1999), 203--228.

\bibitem{Bender-Pokalyuk 2013} C.~Bender and S.~Pokalyuk,
\it Discretization of backward stochastic Volterra integral equations,
\sl Recent Developments in Computational Finance, Interdiscip. Math. Sci., 14,
\rm World Sci. Publ., Hackensack, NJ, 2013, 245--278.


\bibitem{Briand-Hu 2006} P.~Briand and Y.~Hu,
\it BSDE with quadratic growth and unbounded terminal value,
\sl Probab. Theory Related Fields,
\rm {\bf 136} (2006), 604--618.

\bibitem{Briand-Hu 2008} P.~Briand and Y.~Hu,
\it Quadratic BSDEs with convex generators and unbounded terminal conditions,
\sl Probab. Theory Related Fields,
\rm {\bf 141} (2008), 543--567.

\bibitem{Briand-Richou 2017} P.~Briand and A.~Richou,
\it On the uniqueness of solutions to quadratic BSDEs with non-convex generators,
\rm arXiv:1801.00157v1, 2017.

\bibitem{Delbaen-Hu-Bao 2011} F.~Delbaen, Y.~Hu, and X.~Bao,
\it Backward SDEs with superquadratic growth,
\sl Probab. Theory Related Fields,
\rm {\bf 150} (2011), 145--192.

\bibitem{Delbaen-Hu-Richou 2011} F.~Delbaen, Y.~Hu, and A.~Richou,
\it On the uniqueness of solutions to quadratic BSDEs with convex generators and unbounded terminal conditions,
\sl Ann. Inst. Henri Poincar\'{e} Probab. Stat.,
\rm {\bf 47} (2011), 559--574.

\bibitem{Delbaen-Hu-Richou 2015} F.~Delbaen, Y.~Hu, and A.~Richou,
\it On the uniqueness of solutions to quadratic BSDEs with convex generators and unbounded terminal conditions: the critical case,
\sl Discete Contin. Dyn. Syst.,
\rm {\bf 35} (2015), 5447--5465.

\bibitem{Detlefsen-Scandolo 2005} K.~Detlefsen and G.~Scandolo,
\it Conditional and dynamic convex risk measures,
\sl Finance Stoch.,
\rm {\bf 9} (2005), 539--561.

\bibitem{Di Persio 2014} L.~Di Persio,
\it Backward stochastic Volterra integral equation approach to stochastic differential utility,
\sl Int. Electron. J. Pure Appl. Math.,
\rm {\bf8} (2014), 11--15.

\bibitem{Djordjevic-Jankovic 2013} J.~Djordjevi\'{c} and S.~Jankovi\'{c},
\it On a class of backward stochastic Volterra integral equations,
\sl Appl. Math. Lett.,
\rm {\bf26} (2013), 1192--1197.

\bibitem{Djordjevic-Jankovic 2015} J.~Djordjevi\'{c} and S.~Jankovi\'{c},
\it Backward stochastic Volterra integral equations with additive perturbations,
\sl Appl. Math. Comput.,
\rm {\bf 265} (2015), 903--910.

\bibitem{Duffie-Epstein 1992} D.~Duffie and L.~G.~Epstein,
\it Stochastic differential utility,
\sl Econometrica,
\rm {\bf 60} (1992), 353--394.

\bibitem{El Karoui-Peng-Quenez 1997} N.~El.~Karoui, S.~Peng, and M.~C.~Quenez,
\it Backward stochastic differential equations in finance,
\sl Math. Finance,
\rm {\bf7} (1997), 1--71.

\bibitem{Follmer-Schied 2002} H.~F\"{o}llmer and A.~Schied,
\it Convex measures of risk and trading constraints,
\sl Finance Stoch.,
\rm {\bf 6} (2002), 429--447.

\bibitem{Hu 2018} Y.~Hu and B.~{\O}ksendal,
\it  Linear Volterra backward stochastic integral equations,
\sl Stochastic Process. Appl.,
\rm  {\bf 129} (2019), 626--633.


\bibitem{Hu-Tang 2016} Y.~Hu and S.~Tang,
\it Multi-dimensional backward stochastic differential equations of diagonally quadratic generators,
\sl Stochastic Process. Appl.,
\rm {\bf 126} (2016), 1066--1086.

\bibitem{Karatzas-Shreve 2012} I.~Karatzas and S.~Shreve,
\sl Brownian Motion and Stochastic Calculus,
\rm Springer-Verlag, New York, 1988.

\bibitem{Kazamari 1994} N.~Kazamaki,
\sl Continuous Exponential Martingale and BMO,
\sl Lecture Notes in Mathematics, {\bf1579}, \rm Springer-Verlag, Berlin, 1999.

\bibitem{Kobylanski 2000} M.~Kobylanski,
\it Backward stochastic differential equations and partial differential equations with quadratic growth,
\sl Ann. Probab.,
\rm {\bf 28} (2000), 558--602.

\bibitem{Kramkov-Sergio 2016} D.~Kramkov and P.~Sergio,
\it A system of quadratic BSDEs arising in a price impact model,
\sl Ann. Appl. Probab.,
\rm {\bf26} (2016), 794--817.

\bibitem{Kromer-Overbeck 2017} E.~Kromer and L.~Overbeck,
\it Differentiability of BSVIEs and dynamical capital allocations,
\sl Int. J. Theor. Appl. Finance,
\rm {\bf20} (2017), No.07, 1750047.


\bibitem{Lazrak 2004} A.~Lazrak,
\it Generalized stochastic differential utility and preference for information,
\sl Ann. Appl. Probab.,
\rm {\bf 14} (2004), 2149--2175.

\bibitem{Lazrak-Quenez 2003} A.~Lazrak and M.~C.~Quenez,
\it A generalized stochastic differential utility,
\sl Math. Oper. Res.,
\rm {\bf 28} (2003), 154--180.

\bibitem{Lin 2002} J.~Lin,
\it Adapted solution of a backward stochastic nonlinear Volterra integral equation,
\sl Stoch. Anal. Appl.,
\rm {\bf 20} (2002), 165--183.

\bibitem{Lu 2016} W.~Lu,
\it Backward stochastic Volterra integral equations associated with a L\'evy process and applications,
\rm arXiv:1106.6129v2, 2016.

\bibitem{Ma-Yong 1999} J.~Ma and J.~Yong,
\it Forward-Backward Stochastic Differential Equations and Their Applications,
\sl Lecture Notes in Mathematics, {\bf1702},
\rm Springer-Verlag, Berlin, 1999.

\bibitem{Overbeck-Roder p} L.~Overbeck and J.~A.~L.~R\"oder,
\it Path-dependent backward stochastic Volterra integral equations with jumps, differentiability and duality principle,
\rm  Probab. Uncertain. Quant. Risk, (2018).


\bibitem{Pardoux-Peng 1990} E.~Pardoux and S.~Peng,
\it Adapted solution of a backward stochastic differential equation,
\sl Systems Control Lett.,
\rm {\bf14} (1990), 55--61.

\bibitem{Peng 1997} S.~Peng,
\it Backward stochastic differential equations --- stochastic optimizaton theory and viscosity solution of HJB equations,
\sl Topics on Stochastic Analysis, eds. J.~Yan, S.~Peng, S.~Fang, and L.~Wu,
\rm Science Press, Beijing, 1997, 85--138. (in Chinese)

\bibitem{Ren 2010} Y.~Ren,
\it On solutions of backward stochastic Volterra integral equations with jumps in Hilbert spaces,
\sl J. Optim. Theory Appl.,
\rm {\bf 144} (2010), 319--333.

\bibitem{Riedel 2004} F.~Riedel,
\it Dynamic coherent risk measures,
\sl Stochastic Process. Appl.,
\rm {\bf112} (2004), 185--200.

\bibitem{Shi-Wang-Yong 2013} Y.~Shi, T.~Wang, and J.~Yong,
\it Mean-field backward stochastic Volterra integral equations,
\sl Discrete Contin. Dyn. Syst. Ser. B,
\rm {\bf 18} (2013), 1929--1967.

\bibitem{Shi-Wang-Yong 2015} Y.~Shi, T.~Wang, and J.~Yong,
\it Optimal control problems of forward-backward stochastic Volterra integral equations,
\sl Math. Control Relat. Fields,
\rm {\bf5} (2015), 613--649.

\bibitem{Tang 2003} S.~Tang,
\it General linear quadratic optimal stochastic control problems with random coefficients: linear stochastic Hamilton systems and backward stochastic Riccati equations,
\sl SIAM J. Control Optim.,
\rm {\bf 42} (2003), 53--75.

\bibitem{Wang 2018} T.~Wang,
\it Linear quadratic control problems of stochastic Volterra integral equations,
\sl ESAIM: COCV,
\rm {\bf24} (2018), 1849--879.


\bibitem{Wang-Yong 2015} T.~Wang and J.~Yong,
\it Comparison theorems for some backward stochastic Volterra integral equations,
\sl Stochastic Process. Appl.,
\rm {\bf125} (2015), 1756--1798.

\bibitem{Wang-Yong 2018} T.~Wang and J.~Yong,
\it Backward stochastic Volterra integral equations---representation of adapted solutions,
\sl  Stochastic Process. Appl.,
\rm to appear.

\bibitem{Wang-Zhang 2017} T.~Wang and H.~Zhang,
\it Optimal control problems of forward-backward stochastic Volterra integral equations with closed control regions,
\sl SIAM J. Control Optim.,
\rm {\bf55} (2017), 2574--2602.

\bibitem{Wang-Zhang 2007} Z.~Wang and X.~Zhang,
\it Non-Lipschitz backward stochastic Volterra type equations with jumps,
\sl  Stoch. Dyn.,
\rm {\bf7} (2007), 479--496.

\bibitem{Wang-Zhang p} Z.~Wang and X.~Zhang,
\it A class of backward stochastic Volterra integral equations with jumps and applications,
\rm preprint.


\bibitem{Yong 2006} J.~Yong,
\it Backward stochastic Volterra integral equations and some related problems,
\sl Stochastic Process. Appl.,
\rm {\bf116} (2006), 779--795.

\bibitem{Yong 2007} J.~Yong,
\it Continuous-time dynamic risk measures by backward stochastic Volterra integral equations,
\sl Appl. Anal.,
\rm {\bf86} (2007), 1429--1442.

\bibitem{Yong 2008} J.~Yong,
\it Well-posedness and regularity of backward stochastic Volterra integral equations,
\sl Probab. Theory Related Fields,
\rm {\bf 142} (2008), 21--77.

\bibitem{Yong 2012} J.~Yong,
\it Time-inconsistent optimal control problems and the equilibrium HJB equation,
\sl Math. Control Relat. Fields,
\rm {\bf2} (2012), 271--329.

\bibitem{Yong-Zhou 1999} J.~Yong and X.~Y.~Zhou,
\it Stochastic Control: Hamiltonian Systems and HJB Equations,
\rm Springer-Verlag, New York, 1999.

\bibitem{Zhang 2017} J.~Zhang,
\it Backward Stochastic Differential Equations: From Linear to Fully Nonlinear Theory,
\rm Springer-Verlag, New York, 2017.


\end{thebibliography}
\end{document}